\documentclass[final,leqno]{siamltexmm}

\usepackage{graphicx}
\usepackage{tikz}
\usepackage{pgfplots}
\usepackage{colortbl}
\usepackage{amssymb}
\usepackage{amsmath}
\usepackage{bbm}
\usepackage{todonotes}

\usepackage{pdfpages}

\newcommand{\R}{\mathbb R}

\title{The use of discrete gradient methods for total variation type regularization problems in image processing}

\author{V. Grimm\thanks{
Department of Mathematics,
Karlsruhe Institute of Technology (KIT),
	D--76128 Karlsruhe,
        Germany, {\tt email:volker.grimm@kit.edu}.
}
\and R. I. McLachlan\thanks{
Institute of Fundamental Sciences, 
Massey University, Private Bag 11-222, Palmerston North, New Zealand,
{\tt email:r.mclachlan@massey.ac.nz}.}	
\and D. I. McLaren\thanks{Department of Mathematics and Statistics, La Trobe University, Victoria 3086, Australia, {\tt email:\{D.Mclaren, r.quispel\}@latrobe.edu.au}.}
\and 
G. R. W. Quispel\footnotemark[3]
\and C.-B. Sch\"onlieb\thanks{Department of Applied Mathematics and Theoretical Physics (DAMTP), University of Cambridge, Wilberforce Road, Cambridge CB3 0WA, United Kingdom, {\tt email:cbs@cam.ac.uk}.}}	

\begin{document}
\maketitle
\newcommand{\slugmaster}{}%

\begin{abstract}
Discrete gradient methods are well-known methods of Geometric Numerical Integration, which preserve the dissipation of 
gradient systems.  The preservation of the dissipation of a system is an important feature in numerous image processing tasks. We promote the use of discrete gradient methods in image processing by exhibiting experiments with nonlinear total variation (TV) deblurring, denoising, and inpainting.
\end{abstract}



\pagestyle{myheadings}
\thispagestyle{plain}
\markboth{V.~Grimm, R.~McLachlan, D.~McLaren, G.R.W.~Quispel and C.-B.~Sch\"onlieb}{Discrete gradients in image processing}

\section{Introduction} \label{sec:Intro}

The discrete gradient method is mainly interesting as a computational method that preserves the underlying geometric structure of a gradient flow. Since we always end up with a problem in $\R^n$ after discretization, we discuss the properties of the method in $\R^n$ with some inner product $\langle \cdot, \cdot \rangle$. Given a differentiable functional $V\,:\,\R^n\rightarrow \R$, a gradient flow is the solution of the initial value problem
\begin{equation} \label{gradflow}
\dot{x}= - \nabla V (x), \qquad x(0)=x_0\,,
\end{equation}
where the dot represents differentiation with respect to time.
We have the immediate consequence that the functional is nonincreasing along the solution of the evolution equation~\eqref{gradflow}. Or, more exactly, we have the decay
\begin{equation} \label{decaycontinuous}
\frac{d}{dt} V(x(t)) = \langle \nabla V(x(t)), \dot{x} \rangle =  -\|\nabla V(x(t))\|^2 \leq 0\,.
\end{equation}
Gradient systems of this type appear in many areas of image processing, for example, time-marching schemes, nonlinear diffusion filters such as the Perona-Malik model  (cf. \cite{PMscale}) and many variants thereof, Sobolev gradient flows, image registration (e.g. \cite{GriHeWi06, StrDrRu04}) and some applications of active contours, snakes and level sets (e.g. \cite{ChanVese01,MiRatTan07,XuPri98}). In all applications, the preservation of the decay and the limit solution are the most important aspects. It is not so important to solve the evolution equation as accurately as possible, but to find the equilibrium as exactly as possible. Therefore, the preservation of the dissipative behavior of the evolution equation is very important. This is certainly not a new observation and has been expressed by several authors. For example, in \cite{Weickertbook98}, a thorough discussion of the impact of the preservation of dissipativity in a gradient system with respect to diffusion filtering can be found. If the functional $V$ is non-convex, as in some image processing applications such as sparse $\ell_p$ regularization \cite{nikolova2010fast,hintermüller2013nonconvex} and inpainting or data classification with the Ginzburg-Landau energy \cite{bertozzi2007inpainting,burger2009cahn,bertozzi2012diffuse}, then a decay guarantee within a gradient flow formulation will at least guarantee monotone convergence to a critical point of $V$. Structure preservation is the main topic of Geometric Numerical Integration (e.g.~\cite{GNI2, SixLec}), which has been an active research area over the last two decades. Discrete gradient methods turned out to be especially useful to preserve the dissipative structure of a gradient system. We introduce these methods here and show that they have a great potential to be useful for image processing tasks. We also note that discrete gradient methods allow the use of adaptive time steps. Since in most image processing applications the accuracy with respect to the time evolution of the gradient flow is less important than the desire for good descent directions that take the iterates to the equilibrium fast, adaptive time steps seem particularly attractive. This in turn makes it possible to use large time steps initially, followed by smaller time steps once one approaches the equilibrium.  

The paper is organized as follows: The discrete gradient method as a numerical method in Geometric Numerical Integration is introduced in Section~\ref{sec:DGmethod} together with some of their favorable properties in Section~\ref{sec:PropDG}. In Section~\ref{sec:nonlinexmpl} several experiments with well-known gradient systems in image processing are conducted that illustrate the importance of the preservation of the dissipativity of a gradient system and that indicate that these methods might be useful in image processing.
Finally, a brief conclusion is given in Section~\ref{sec:conclusion}.

\section{Discrete gradient method} \label{sec:DGmethod}
For simplicity, we will use $\R^n$ equipped with the standard inner product and its associated norm. The proofs for this case can be generalized to Hilbert spaces, but for our purposes, where we either think of a digital grayscale image taken by a digital camera, $\R^n$ where $n=N_xN_y$ and $N_x$ and $N_y$ correspond to the number of pixels with respect to width and height of the digitized picture, or an $RGB$ picture where $n=3N_xN_y$ and $N_x$ and $N_y$ correspond to the number of pixels with respect to width and height of the digitized picture for the red, green, and blue color channel will be sufficient.

\begin{definition} \label{defdisgrad}
Let $V\,:\,\R^n \rightarrow \R$ be continuously differentiable. The function $\overline{\nabla}V\,:\, \R^n\times \R^n \rightarrow \R^n$ is a \emph{discrete gradient} of $V$ iff it is continuous and
\[
\left\{
  \begin{array}{rcl}
   \langle \overline{\nabla}V (x,x'), (x'-x)\rangle &=& V(x')-V(x), \\
    \overline{\nabla}V (x,x) &=& \nabla V(x)\,,
  \end{array}  
\right. \qquad \mathrm{for~all} \quad x, x' \in \R^n\,. 
\]
\end{definition}
Note, that this definition is different from what is often understood to be a discrete gradient in image processing where the term just refers to a discretized gradient. Definition~\ref{defdisgrad} asks for a specific condition.
Discrete gradients according to this definition have been studied by many researchers in Geometric Numerical Integration (e.g. \cite{Cie13, CoHai11, AVF12, Gonzales96, HaiLub13, MatFuri01, McLachlanQusipel99, QuiTu96, StuHam96}). 
Three well-known discrete gradients are the \emph{midpoint discrete gradient} or \emph{Gonzalez discrete gradient}
(cf. \cite{Gonzales96})
\begin{equation}\label{eq:DGgonzalez}
\begin{array}{rcl}
  \overline{\nabla}_1 V(x,x') &=& \nabla V 
    \left( 
      \frac{x'+x}{2}
    \right) \\[1ex] 
   & & + \,
  \frac{
    V(x')-V(x)-
        \left\langle \nabla V 
	   \left(\frac{x'+x}{2} \right),x'-x \right\rangle
  }{
    \|x-x'\|^2
  }(x'-x)\,, \qquad (x \not= x'),
\end{array}
\end{equation}
the \emph{mean value discrete gradient}
\[
\overline{\nabla}_2 V(x,x') = \int_0^1 \nabla V((1-s)x+s x')\,ds, \quad 
\]
that is, for example, used in the averaged vector field method (cf.~\cite{AVF12}), and the discrete gradient proposed by Itoh \& Abe (cf. \cite{ItohAbe88}) that reads
\renewcommand{\arraystretch}{1.75}
\begin{align}\label{itoabe:discretegrad}
\overline{\nabla}_3 V(x,x') &= \left(\begin{array}{c}
\frac{V(x_1',x_2,\ldots,x_n)-V(x)}{ x_1'-x_1}\\
\frac{V(x_1',x_2',x_3,\ldots,x_n)-V(x_1',x_2,\ldots,x_n)}{x_2'-x_2}\\
\vdots\\
\frac{ V (x_1',\ldots, x_{n-1}',x_n)-V( x_1',\ldots, x_{n-2}',x_{n-1},x_n)}{\tilde x_{n-1}-x_{n-1}}\\
\frac{V(x')-V(x_1',\ldots,x_{n-1}',x_n)}{x_n'-x_n}
\end{array}\right), \qquad x_i' \not= x_i, \quad i=1,\ldots,n,
\end{align}
\renewcommand{\arraystretch}{1}
Note that the Itoh \& Abe discrete gradient is derivative-free and hence its computational realization relatively cheap. Besides these discrete gradients, there are many more. For the gradient flow
\begin{equation} \label{gradsys}
 \dot{x}=- \nabla V(x), \qquad x(0)=x_0
\end{equation}
every discrete gradient $\overline{\nabla }V$ leads to an associated discrete gradient method
\begin{equation} \label{dgmethod}
x_{n+1}-x_n=-\tau_n \overline{\nabla}V(x_n,x_{n+1})\,,
\end{equation}
where $\tau_n > 0$ is a time step that might vary from step to step. Due to Definition~\ref{defdisgrad} of a discrete gradient, this method preserves the dissipativity of the solution of the gradient system~\eqref{gradflow}, that is we have
\begin{eqnarray*}
V(x_{n+1})-V(x_n)&=&\langle \overline{\nabla}V(x_n,x_{n+1}),(x_{n+1}-x_n)\rangle =
-\tau_n \|\overline{\nabla}V(x_n,x_{n+1})\|^2 \leq 0
\end{eqnarray*}
for all steps $n$ and arbitrary $\tau_n > 0$ as a discrete analogue to the decay \eqref{decaycontinuous} of the continuous solution. 

For our numerical illustrations, we will mainly use the Gonzalez and Itoh-Abe discrete gradient. But we would like to stress that the properties just described as well as the theoretical results in the following sections hold for arbitrary discrete gradients - a rich family to pick from.

\section{Some properties of discrete gradient methods} \label{sec:PropDG}
The preservation of the dissipativity by a discrete gradient method leads to useful consequences. Before we can state our first result, we need to recall some definitions. 
\begin{definition} 
A functional $V\,:\,\R^n \rightarrow \R$ is called 
\begin{itemize}
\item {\em coercive} iff
\[
V(x_n) \rightarrow \infty \qquad \mbox{for} \qquad \|x_n\| \rightarrow \infty.
\]
\item {\em bounded from below} iff there exists a constant $C$ such that
\[
 C \leq V(x), \qquad \mbox{for all}\quad x \in \R^n\,.
\]
\item {\em convex} iff for all $x, y \in \R^n$ and $\lambda \in [0,1]$
\[
V(\lambda x +(1-\lambda)y) \leq \lambda V(x)+(1-\lambda)V(y)\,.
\]
\item {\em strictly convex} iff for all $x,y \in \R^n$, $x \not= y$, and $\lambda \in (0,1)$
\[
V(\lambda x +(1-\lambda)y) < \lambda V(x)+(1-\lambda)V(y)\,.
\]
\end{itemize}
\end{definition}
\begin{theorem} \label{seqtostat}
Let $\nabla V$ in \eqref{gradsys} stem from a functional $V$ which is bounded from below, coercive and continuously differentiable. If $\left\{ x_n \right\}_{n=0}^\infty$ is a sequence generated by the discrete gradient method \eqref{dgmethod} with time steps 
$0 < c \leq \tau_n \leq M < \infty$, then 
\[
  \lim_{n \rightarrow \infty} \overline{\nabla} V(x_{n+1},x_n)=
  \lim_{n \rightarrow \infty} \nabla V (x_n)=0\,.
\]
There exists at least one accumulation point of the sequence $\left\{ x_n \right\}_{n=0}^\infty$. And for any accumulation point $x_*$ of the sequence $\left\{ x_n \right\}_{n=0}^\infty$, we have $\nabla V(x_*)=0$.
\end{theorem}

\begin{proof} Since $V$ is bounded from below, say by $C$, and due to the preservation of the dissipativity, we find
\[
   C \leq V(x_{n+1}) \leq V(x_n) \leq \cdots \leq V(x_0), \qquad n=1,2,3,\ldots
\]
and hence the limit
\[
  \lim_{n \rightarrow \infty} V(x_n) = V_*
\]
exists. From Definition~\ref{defdisgrad} and the definition of the discrete gradient method in \eqref{dgmethod}, we find
\begin{align*}
\tau_n\|\overline{\nabla}V(x_{n+1},x_n)\|^2 
&=
-\langle \overline{\nabla}V(x_{n+1},x_n), x_{n+1}-x_n \rangle 
= V(x_n)-V(x_{n+1}) \\
&=
\frac{1}{\tau_n} \langle -\tau_n\overline{\nabla}V(x_{n+1},x_n), x_{n+1}-x_n \rangle \\
&=
\frac{1}{\tau_n} \| x_{n+1}-x_n \|^2 \geq 0
\end{align*}
for all $n$. By summing these equations from $n$ to $m-1$, $ m > n$, we obtain
\[
\sum_{k=n}^{m-1} \tau_k \left\|\overline{\nabla}V(x_{k+1},x_k)\right\|^2
  =\sum_{k=n}^{m-1} \frac{1}{\tau_k} \left\|x_{k+1}-x_k \right\|^2
  =V(x_n)-V(x_m)  \leq V(x_0)-V_*\,
\]
and thus
\[
\sum_{k=0}^\infty \left\|\overline{\nabla}V(x_{k+1},x_k)\right\|^2 \leq \frac{V(x_0)-V_*}{c} < \infty, \qquad 
\sum_{k=0}^\infty \left\|x_{k+1}-x_k \right\|^2 \leq M \left( V(x_0) -V_* \right)< \infty
\]
and therefore
\[
  \lim_{n \rightarrow \infty} (x_{n+1}-x_n)=\lim_{n \rightarrow \infty} \overline{\nabla}V(x_{n+1},x_n)=0\,.
\]
The sets $V_t$ defined by
\[
  V_t := \left\{ x \in \R^n~|~ V(x) \leq t \right\} 
\]
are empty or compact. Hence the set $V_{V(x_0)}$ is bounded and closed. In particular, $\overline{\nabla}V$ is uniformly continuous on $V_{V(x_0)} \times V_{V(x_0)}$, where we have chosen the usual topology on the product space to coincide with the norm induced by the standard scalar product on $\R^{2n}$. Therefore, for any $\epsilon > 0$ there exists a $\delta$ such that for $\|(x_{n+1},x_n)-(x_n,x_n)\| = \|x_{n+1}-x_n\| \leq \delta$ we have
\[
   \|\overline{\nabla} V(x_{n+1},x_n)-\overline{\nabla} V(x_n,x_n)\|=\|\overline{\nabla} V(x_{n+1},x_n)-\nabla V(x_n)\| < \epsilon.
\]
Since $\lim_{n \rightarrow \infty} \|x_{n+1}-x_n\|=0$, we find, that for large enough $n$, we have
\[
  \| \nabla V(x_n)\| \leq \| \overline{\nabla} V(x_{n+1},x_n)-\nabla V(x_n)\| + \|\overline{\nabla} V(x_{n+1},x_n)\| \leq 2\epsilon\,.
\]
Hence, altogether, we conclude
\[
  \lim_{n \rightarrow \infty} \overline{\nabla} V(x_{n+1},x_n)=\lim_{n \rightarrow \infty} \nabla V (x_n)=0\,.
\]
Due to the boundedness of $V_{V(x_0)}$, the sequence $\left\{ x_n \right\}_{n=0}^\infty$ has at least one accumulation point $x_*$ by the Bolzano-Weierstrass theorem. For a subsequence $\left\{ x_{n_l} \right\}_{l=0}^\infty$ with $\lim_{l \rightarrow \infty} x_{n_l}=x_*$, we have
\[
  0=\lim_{l \rightarrow \infty} \nabla V(x_{n_l})=\nabla V(x_*),\,
\]
due to the continuity of $\nabla V$.
\end{proof}

Theorem~\ref{seqtostat} states that the sequence $\left\{ x_n \right\}_{n=0}^\infty$ generated by any discrete gradient method satisfies $\lim_{n \rightarrow \infty} \nabla V(x_n)=0$. This property is very important in the minimization of functionals. In image processing, the functionals to be minimized are often convex or even strictly convex. For such functionals any discrete gradient method tends to global minimizers. 
\begin{theorem}
Under the assumptions of Theorem~\ref{seqtostat}.
\begin{enumerate}
\item If $V$ is in addition convex, then a minimizer exists and any accumulation point of the sequence $\left\{ x_n \right\}_{n=0}^\infty$ is a minimizer.
\item If $V$ is in addition strictly convex, then 
\[
  \lim_{n \rightarrow \infty} x_n=x_*, \qquad V(x_*)=\min_{x} V(x)\,,
\]
that is, the sequence of the discrete gradient approximations converges to the unique minimizer.
\end{enumerate}
\end{theorem}
\begin{proof}
It is a standard result, that a continuously differentiable function $V\,:\,\R^n \rightarrow \R$ is 
convex, if and only if
\[
  V(u)+\langle \nabla V(u), w-u \rangle \leq V(w), \qquad \mbox{for~all} \quad u,w \in \R^n\,.
\]
For an accumulation point $x_*$ of the sequence $\left\{ x_n \right\}_{n=0}^\infty$ generated by the discrete gradient method, we have $\nabla V(x_*)=0$ according to Theorem~\ref{seqtostat} and therefore $V(x_*) \leq V(w)$, for all $w \in \R^n$ which means that $x_*$ is a minimizer of the functional $V$. There is at least one accumulation point of the sequence $\left\{ x_n \right\}_{n=0}^\infty$ according to Theorem~\ref{seqtostat} and therefore a minimizer exists.

Assume that the function $V$ is strictly convex and that $x_*$ and $y_*$ were two different minimizers of $V$, that is  $x_* \not= y_*$ and $V(x_*)=V(y_*) \leq V(w)$ for all $w \in \R^n$. Since $V$ is strictly convex, pick $\lambda \in (0,1)$ and we obtain
\[
  V(\lambda x_* + (1-\lambda) y_*) < \lambda V(x_*) + (1-\lambda)V(y_*)=\lambda V(x_*) + (1-\lambda)V(x_*)=V(x_*)\,.
\]
Since $\lambda x_* + (1-\lambda) y_* \in \R^n$, this is a contradiction to $x_*$ (or $y_*$, respectively) being a minimizer. Hence the minimizer must be unique and therefore all accumulation points of the sequence $\left\{ x_n \right\}_{n=0}^\infty$, which are minimizers, must be identical. Therefore, the sequence converges to the unique minimizer. 
\end{proof}

\section{Nonlinear Examples} \label{sec:nonlinexmpl}

In this section, we illustrate the positive effect of the preservation of dissipativity by a series of numerical experiments on standard models in image processing, that involve a gradient flow.
In Subsection~\ref{sec:TVdenoising}, we study the TV denoising (also TV cartooning) functional, whose discretized version is strictly convex. The theory of Section~\ref{sec:PropDG} is applicable and we illustrate numerically with the Gonzalez discrete gradient that the discrete gradient method shows the predicted behavior. As a simple method that does not possess the preservation of decay property, the explicit Euler method is used for comparison. 
The introduction of a blurring kernel in the functional in Subsection~\ref{sec:TVdenoising}, which leads to a deblurring example, shows the same good effects of the preservation of the decay of the functional. After these basic examples, we provide three more experiments that generalize the application of discrete gradient methods in different ways. In Subsection~\ref{sec:TVinpainting} we apply the discrete gradient method to solve TV image inpainting. We discuss its performance using the Itoh \& Abe gradient with a simple adaptive step size rule and compare it with the so-called lagged diffusivity method \cite{AcVog94}, which for convex functionals $V$ shares the dissipation property of the discrete gradient approach \cite{chan1999convergence}. A question that is always important to answer is whether newly proposed methods are useful in actual applications. We therefore study a real-world color denoising example in Subsection~\ref{sec:TVmultichannel}.
Finally, in Subsection~\ref{sec:TVnonconvex} we study non-convex TV denoising \cite{hintermüller2013nonconvex} computed with the Itoh \& Abe discrete gradient method. We include this example to demonstrate the flexibility of the discrete gradient, and its robust structure-preserving properties, which guarantee monotonic decay of $V$ even in the non-convex case.

In what follows, we denote by $T_\alpha$ the continuous functional and by $V_\alpha$ its discretization for which the discrete gradient is computed. The parameter $\alpha$ indicates the dependence of the functional on a positive parameter $\alpha$ that weights the TV regularization against a fitting term to the given image data.


\subsection{Grayscale image denoising with TV regularization}\label{sec:TVdenoising}
For the denoising of a grayscale image, the following gradient descent method has been proposed in \cite{RuOshFa92}. The gradient system is based on the  functional
  \begin{equation}
    T_\alpha(u)= \frac{1}{2} \int_\Omega (u(x,y)-u_0(x,y))^2\,d(x,y) + \alpha TV(u),
  \end{equation}
  where we use the smoothed TV functional  
  \begin{equation} \label{TVfunctgraycont}
    TV(u)=\int_\Omega \sqrt{\left( \frac{\partial u}{\partial x}\right)^2
    + \left( \frac{\partial u}{\partial y}\right)^2 + \beta} \,\, d(x,y),
  \end{equation}
  with parameter $0<\beta \ll 1$ as suggested in \cite{AcVog94}. Under the smoothness assumption $u \in W^{1,1}$, this leads to the gradient system
  \[
  u_t = \alpha \nabla \cdot \left[ \frac{\nabla u}{|\nabla u|_\beta} \right]
    -(u-u_0) \qquad 
    \left. \frac{\partial u}{\partial {\bf n}} \right|_{\partial \Omega} = 0
  \]
  where $|\nabla u|_\beta = \sqrt{|\nabla u|^2+\beta}$.  
  For the computation, discretization of the continuous functional is necessary.   As in \cite{RuOshFa92}, we use finite differences, where the homogeneous Neumann boundary conditions are discretized by duplicating the boundary rows and columns of the original picture array. The discretized functional reads
\begin{equation} \label{functdenograycont}
    V_\alpha(u)=\frac{1}{2} \Delta x \Delta y \sum_{i=1}^{N_x} \sum_{j=1}^{N_y} \left( u_{i,j} - (u_0)_{i,j} \right)^2+ \alpha J(u),
\end{equation}
  with  
\begin{equation} \label{TVfunctgray}
     J(u)=\Delta x \Delta y \sum_{i=1}^{N_x} \sum_{j=1}^{N_y} \psi \left( (D^x_{ij}u)^2 
     + (D^y_{ij}u)^2 \right),
\end{equation}
  where $\psi(t)=\sqrt{t+\beta}$ and
  \[
    \qquad D_{ij}^xu = \frac{u_{i,j}-u_{i-1,j}}{\Delta x}, 
    \qquad \qquad D_{ij}^yu = \frac{u_{i,j}-u_{i,j-1}}{\Delta y}.
  \]
Here $u \in \R^{N_x} \times \R^{N_y}$ is the discretized picture. As usual, we identify the matrix 
$u\in \R^{N_x} \times \R^{N_y}$ with the vector $u\in \R^{N_xN_y}$ by running successively through the columns of $u \in \R^{N_x} \times \R^{N_y}$. The gradient system in $\R^{N_xN_y}$ then follows analogously to the continuous system by computing the gradient:
\begin{equation*} 
  \dot{u} = -\nabla V_\alpha(u), \qquad u(0)=u_0\,.
\end{equation*}
This gradient system satisfies the assumption of our theorems in Section~\ref{sec:PropDG}.
\begin{lemma} \label{Denoisegraylem}
The discretized functional $V_\alpha$ in \eqref{functdenograycont} is bounded from below, coercive, continuously differentiable and strictly convex. 
\end{lemma}
\begin{proof} Due to
{\small
\begin{align*}
   J(\lambda u &+ (1-\lambda)v) \\
   & = \Delta x \Delta y \sum_{i=1}^{N_x} \sum_{j=1}^{N_y} \sqrt{ 
   \left( D_{ij}^x (\lambda u +(1-\lambda)v) \right)^2   
   +
   \left( D_{ij}^y (\lambda u +(1-\lambda)v) \right)^2   
   +
   \left( \lambda \sqrt{\beta} + (1-\lambda) \sqrt{\beta} \right)^2
   } \\
& = \Delta x \Delta y \sum_{i=1}^{N_x} \sum_{j=1}^{N_y}
   \sqrt{ 
   \left( \lambda D_{ij}^x u +(1-\lambda) D_{ij}^x v  \right)^2   
   +
   \left( \lambda D_{ij}^y u +(1-\lambda) D_{ij}^y v \right)^2   
   +
   \left( \lambda \sqrt{\beta} + (1-\lambda) \sqrt{\beta} \right)^2
   } \\   
& \leq \Delta x \Delta y
  \sum_{i=1}^{N_x} \sum_{j=1}^{N_y} \sqrt{ 
   \left( \lambda D_{ij}^x u \right)^2   
   +
   \left( \lambda D_{ij}^y u \right)^2   
   +
   \left( \lambda \sqrt{\beta} \right)^2
   } \\
 & \qquad  \qquad \qquad  + \Delta x \Delta y \sum_{i=1}^{N_x} \sum_{j=1}^{N_y}
  \sqrt{ 
   \left( (1-\lambda) D_{ij}^x v \right)^2   
   +
   \left( (1-\lambda) D_{ij}^y v \right)^2   
   +
   \left( (1-\lambda) \sqrt{\beta} \right)^2
   }  \\
&= \lambda J(u) + (1-\lambda) J(v)           
\end{align*}
}
for $\lambda \in [0,1]$ and $u,v$ two pictures, we find, that $J$ is convex. With two different constant pictures, it is easy to see that $J$ is not strictly convex. By checking that for $u \not= v$ and $t \in [0,1]$, we have
\[
  F(t)=\frac{1}{2} \Delta x \Delta y \|tu+(1-t)v-u_0\|^2, \qquad F''(t)=\Delta x \Delta y \|u-v\|^2>0
\]
and hence the functional $\frac{1}{2} \Delta x \Delta y \|u-u_0\|^2$ is strictly convex. Therefore, $V_\alpha$ as a whole is strictly convex. From this first functional, coercivity is obvious.
\end{proof}

Form Lemma~\ref{Denoisegraylem}, we immediately conclude the following corollary of our theorems.
\begin{corollary} The function $V_\alpha$ has a unique minimizer and the sequence generated by any discrete gradient method with step sizes $0 < c \leq \tau_n \leq M < \infty$ converges to the minimizer.
\end{corollary}

\begin{figure}
\begin{center}
\begin{tikzpicture}
\begin{semilogxaxis}[width=15cm,height=7cm,xmin=0.75,xmax=11000,ymin=300,ymax=500, xtick={1,10,100, 1000, 10000}, ytick={350,400,450}, xticklabels={$1$, $10$, $100$, $1000$, $10000$},  yticklabels={$350$, $400$, 
$450$}, title={}, xlabel=\empty, ylabel=\empty]
\addplot[thick,color=cyan, dotted] coordinates {
( 1, 467.8589 )
( 2, 374.2165 )
( 3, 414.8694 )
( 4, 419.3028 )
( 6, 446.0137 )
( 8, 458.8466 )
( 10, 464.3845 )
( 20, 476.5267 )
( 30, 480.9473 )
( 40, 483.4359 )
( 50, 484.9022 )
( 60, 486.1986 )
( 70, 487.026 )
( 80, 487.8055 )
( 90, 488.3066 )
( 100, 488.7647 )
( 200, 490.4434 )
( 300, 490.8837 )
( 400, 491.2657 )
( 500, 491.4068 )
( 600, 491.4569 )
( 700, 491.4747 )
( 800, 491.4881 )
( 900, 491.4884 )
( 1000, 491.5005 )
( 2000, 491.5006 )
( 3000, 491.5006 )
( 4000, 491.5006 )
( 5000, 491.5006 )
( 6000, 491.5006 )
( 7000, 491.5006 )
( 8000, 491.5006 )
( 9000, 491.5006 )
( 10000, 491.5006 )
};
\addplot[thick,color=red, dashed] coordinates {
( 1, 467.8589 )
( 2, 365.5793 )
( 3, 364.4971 )
( 4, 368.401 )
( 6, 385.2175 )
( 8, 393.4588 )
( 10, 397.5097 )
( 20, 405.0403 )
( 30, 407.7537 )
( 40, 409.3211 )
( 50, 410.3721 )
( 60, 411.0359 )
( 70, 411.5584 )
( 80, 411.9469 )
( 90, 412.1417 )
( 100, 412.3218 )
( 200, 413.3989 )
( 300, 413.687 )
( 400, 413.788 )
( 500, 413.8501 )
( 600, 413.8829 )
( 700, 413.8832 )
( 800, 413.8845 )
( 900, 413.885 )
( 1000, 413.885 )
( 2000, 413.8875 )
( 3000, 413.8875 )
( 4000, 413.8875 )
( 5000, 413.8875 )
( 6000, 413.8875 )
( 7000, 413.8875 )
( 8000, 413.8875 )
( 9000, 413.8875 )
( 10000, 413.8875 )
};
\addplot[thick,color=red, dashed] coordinates {
( 1, 467.8589 )
( 2, 371.1018 )
( 3, 345.0953 )
( 4, 338.6015 )
( 6, 342.2888 )
( 8, 346.7796 )
( 10, 349.3206 )
( 20, 353.6875 )
( 30, 354.9136 )
( 40, 355.5669 )
( 50, 356.0545 )
( 60, 356.4045 )
( 70, 356.7088 )
( 80, 356.939 )
( 90, 357.1253 )
( 100, 357.2579 )
( 200, 357.8246 )
( 300, 357.9327 )
( 400, 357.9931 )
( 500, 358.0084 )
( 600, 358.036 )
( 700, 358.0372 )
( 800, 358.0372 )
( 900, 358.0372 )
( 1000, 358.0372 )
( 2000, 358.0371 )
( 3000, 358.0371 )
( 4000, 358.0371 )
( 5000, 358.0371 )
( 6000, 358.0371 )
( 7000, 358.0371 )
( 8000, 358.0371 )
( 9000, 358.0371 )
( 10000, 358.0371 )
};
\addplot[thick,color=red, dashed] coordinates {
( 1, 467.8589 )
( 2, 391.0072 )
( 3, 356.1993 )
( 4, 339.0017 )
( 5, 330.1571 )
( 6, 325.4548 )
( 7, 323.1111 )
( 8, 321.8375 )
( 9, 321.4435 )
( 10, 321.1641 )
( 20, 322.3839 )
( 30, 322.8598 )
( 40, 323.0612 )
( 50, 323.1602 )
( 60, 323.2291 )
( 70, 323.2769 )
( 80, 323.3061 )
( 90, 323.3213 )
( 100, 323.3362 )
( 200, 323.4645 )
( 300, 323.5303 )
( 400, 323.5485 )
( 500, 323.5525 )
( 600, 323.5554 )
( 700, 323.5576 )
( 800, 323.5576 )
( 900, 323.5576 )
( 1000, 323.5576 )
( 2000, 323.5596 )
( 3000, 323.5596 )
( 4000, 323.5596 )
( 5000, 323.5596 )
( 6000, 323.5596 )
( 7000, 323.5596 )
( 8000, 323.5596 )
( 9000, 323.5596 )
( 10000, 323.5596 )
};
\addplot[thick,color=green, dashdotted] coordinates {
( 1, 467.8589 )
( 2, 424.1207 )
( 3, 394.3943 )
( 4, 373.7773 )
( 5, 359.1891 )
( 6, 348.693 )
( 7, 341.0426 )
( 8, 335.4099 )
( 9, 331.2281 )
( 10, 328.1013 )
( 20, 318.8703 )
( 30, 318.1637 )
( 40, 318.0995 )
( 50, 318.0932 )
( 60, 318.0926 )
( 70, 318.0925 )
( 80, 318.0925 )
( 90, 318.0925 )
( 100, 318.0925 )
( 200, 318.0925 )
( 300, 318.0925 )
( 400, 318.0925 )
( 500, 318.0925 )
( 600, 318.0925 )
( 700, 318.0925 )
( 800, 318.0925 )
( 900, 318.0925 )
( 1000, 318.0925 )
( 2000, 318.0925 )
( 3000, 318.0925 )
( 4000, 318.0925 )
( 5000, 318.0925 )
( 6000, 318.0925 )
( 7000, 318.0925 )
( 8000, 318.0925 )
( 9000, 318.0925 )
( 10000, 318.0925 )
};
\addplot[thick,color=green, dashdotted] coordinates {
( 1, 467.8589 )
( 2, 463.0846 )
( 3, 458.4706 )
( 4, 454.0125 )
( 5, 449.7056 )
( 6, 445.5453 )
( 7, 441.5271 )
( 8, 437.6464 )
( 9, 433.8987 )
( 10, 430.2795 )
( 20, 400.2377 )
( 30, 378.8842 )
( 40, 363.5145 )
( 50, 352.3067 )
( 60, 344.0406 )
( 70, 337.8864 )
( 80, 333.268 )
( 90, 329.7786 )
( 100, 327.1265 )
( 200, 318.8885 )
( 300, 318.1738 )
( 400, 318.1014 )
( 500, 318.0935 )
( 600, 318.0926 )
( 700, 318.0925 )
( 800, 318.0925 )
( 900, 318.0925 )
( 1000, 318.0925 )
( 2000, 318.0925 )
( 3000, 318.0925 )
( 4000, 318.0925 )
( 5000, 318.0925 )
( 6000, 318.0925 )
( 7000, 318.0925 )
( 8000, 318.0925 )
( 9000, 318.0925 )
( 10000, 318.0925 )
};
\addplot[thick,color=green, dashdotted] coordinates {
( 1, 467.8589 )
( 2, 467.3778 )
( 3, 466.8983 )
( 4, 466.4204 )
( 5, 465.9442 )
( 6, 465.4695 )
( 7, 464.9964 )
( 8, 464.5249 )
( 9, 464.0551 )
( 10, 463.5868 )
( 20, 458.9899 )
( 30, 454.5467 )
( 40, 450.2526 )
( 50, 446.1031 )
( 60, 442.0938 )
( 70, 438.2204 )
( 80, 434.4783 )
( 90, 430.8633 )
( 100, 427.3711 )
( 200, 398.3285 )
( 300, 377.6124 )
( 400, 362.6559 )
( 500, 351.7214 )
( 600, 343.6393 )
( 700, 337.6105 )
( 800, 333.0785 )
( 900, 329.6488 )
( 1000, 327.0381 )
( 2000, 318.89 )
( 3000, 318.1748 )
( 4000, 318.1016 )
( 5000, 318.0936 )
( 6000, 318.0926 )
( 7000, 318.0925 )
( 8000, 318.0925 )
( 9000, 318.0925 )
( 10000, 318.0925 )
};
\addplot[thick,color=blue] coordinates {
( 1, 467.8589 )
( 2, 359.8127 )
( 3, 339.7588 )
( 4, 329.8056 )
( 5, 325.1568 )
( 6, 322.3599 )
( 7, 320.7955 )
( 8, 319.823 )
( 9, 319.2155 )
( 10, 318.8502 )
( 20, 318.1145 )
( 30, 318.0938 )
( 40, 318.0929 )
( 50, 318.0929 )
( 60, 318.0929 )
( 70, 318.0929 )
( 80, 318.0929 )
( 90, 318.0929 )
( 100, 318.0929 )
( 10000, 318.0929 )
};
\end{semilogxaxis}
\end{tikzpicture} 
\end{center}
\vspace{-1cm}
\caption{TV functional Baboon, explicit Euler method with step size $\tau=0.5$ (cyan dotted line), explicit Euler method with $\tau=0.4, 0.3, 0.2$ (red dashed lines, top to bottom), explicit Euler with $\tau=0.1, 0.01, 0.001$ (green dash-dotted lines, left to right) discrete gradient method (blue) with step size $\tau = 2.5$  (in all experiments: $\alpha=0.05$, $\beta=0.001$)}
\label{baboondenoisepic}
\end{figure}
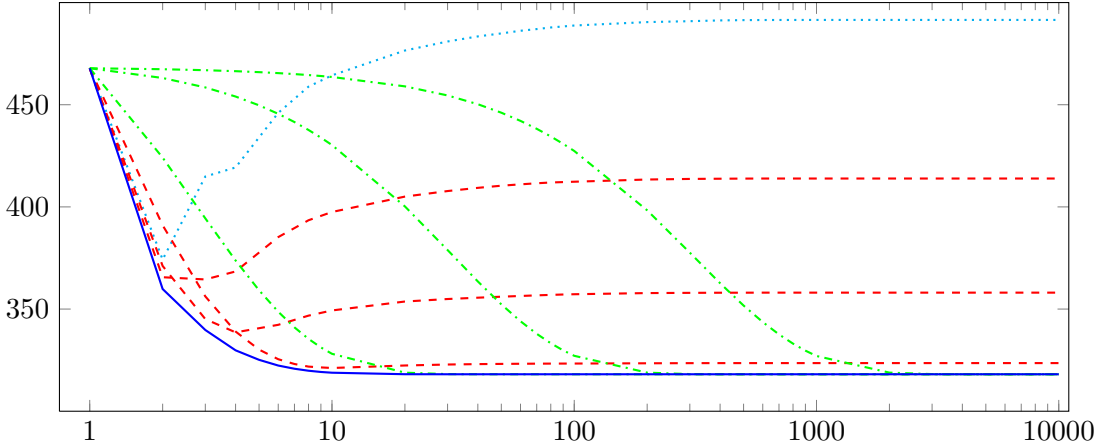

\begin{figure}  
  \begin{center}
  \includegraphics[width=6cm]{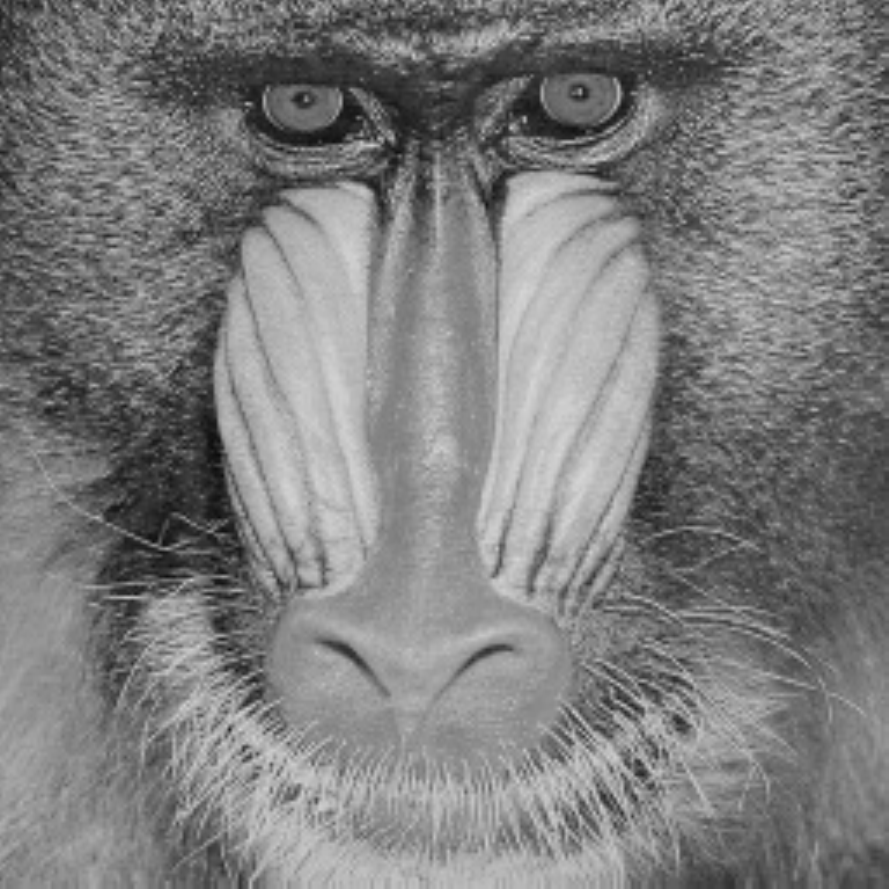}
  \hspace{2cm}
  \includegraphics[width=6cm]{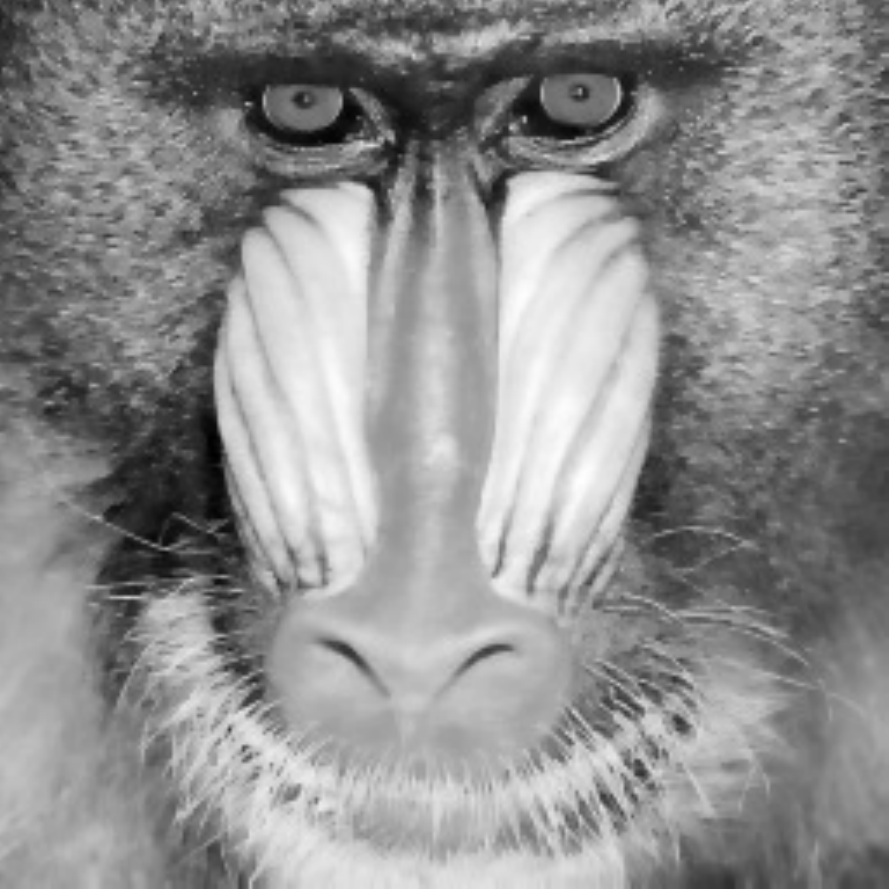}
  \end{center}
  
  \vspace{1cm}
  
  \begin{center}
  \includegraphics[width=6cm]{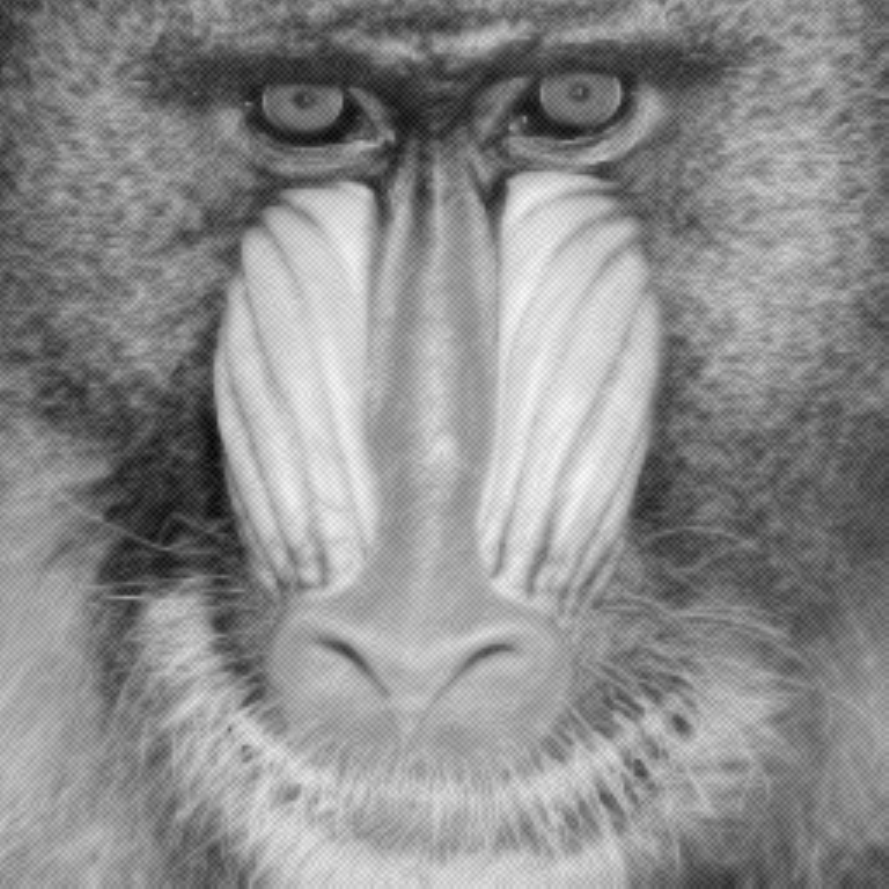}
  \hspace{2cm}
  \includegraphics[width=6cm]{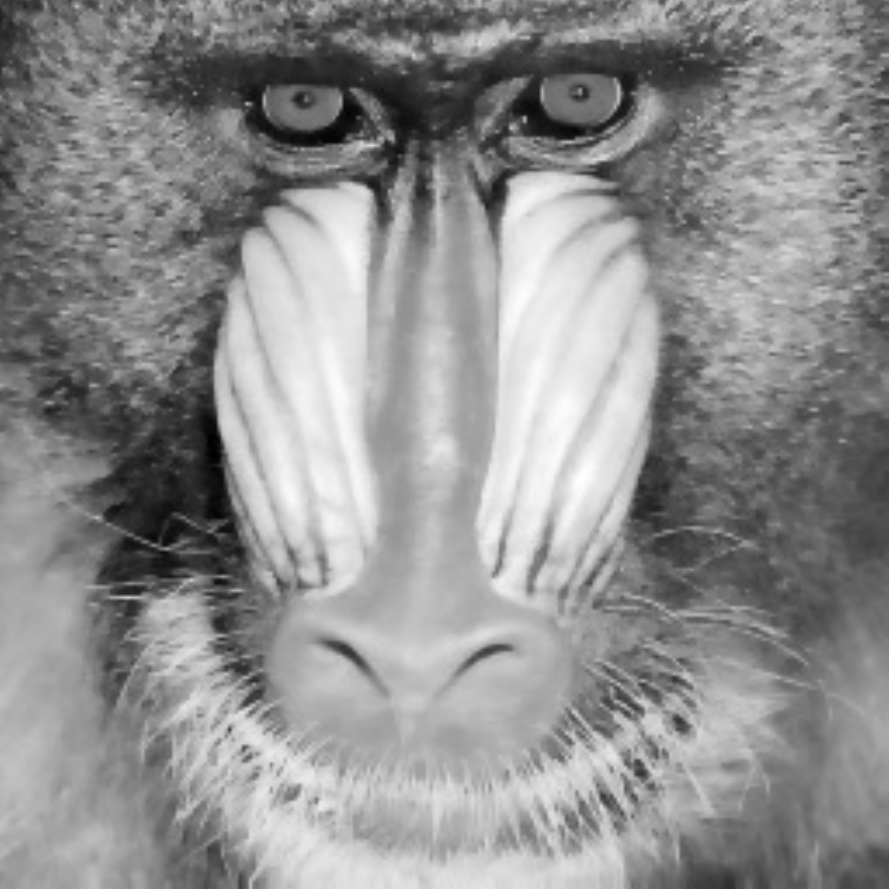}
  \end{center}
  \caption{Original picture (top left) and TV-denoised reference picture with the explicit Euler method, step-size $\tau=0.001$ after $10000$ steps (top right), 
  explicit Euler denoised image with step-size $\tau=0.3$ after $10000$ steps (bottom left) and discrete gradient denoised baboon with step-size $\tau=2.5$ after 10 steps (bottom right).}  \label{expldgdenoisedbaboon}
\end{figure}

In our examples below, we adopted another simplification which is common in image processing. The domain of the picture is scaled into a rectangular region such that $\Delta x=1$ and $\Delta y=1$. Then image data is scaled to $[0,1]$. In all experiments we chose $\alpha=0.05$ and $\beta=0.001$. A first experiment with the Gonzalez discrete gradient shows that the discrete gradient method converges reliably to the minimizer of the total variation functional. In Figure~\ref{baboondenoisepic} the value of the functional is plotted against the number of steps.
For the step size $\tau=2.5$, the discrete gradient method converges in about $10$ steps to the same minimum value of the functional as the explicit Euler method in $40$ steps with step size $\tau=0.1$ or in $200$ steps with step size $\tau=0.01$ or in $2000$ steps with step size $\tau=0.001$. The equilibrium computed by the Euler method with step size $\tau=0.001$ after $10000$ steps is used as reference equilibrium picture. The explicit Euler method with step sizes $\tau=0.2$ and larger stagnates at a larger value of the total variation functional and never converges to the correct value (cf. Figure~\ref{baboondenoisepic}). For these step sizes, the value of the functional oscillates, which can not be seen in the figure due to aliasing. The discrete gradient method cannot oscillate, since it is strictly decreasing.

In our experiment, we used the mandrill a. k. a. baboon picture of the  USC-SIPI Image Database, \cite{USC-SIPIbase}, turned into a grayscale image. At a close inspection, the difference of the final denoised pictures shown in Figure~\ref{expldgdenoisedbaboon} can be seen. The denoised image computed with the Euler method with step size $\tau=0.3$ (Figure~\ref{expldgdenoisedbaboon}, bottom left-hand side) shows slightly more details than the reference picture (Figure~\ref{expldgdenoisedbaboon}, top right-hand side), while the denoised image computed with the discrete gradient method at step size $\tau=2.5$ (Figure~\ref{expldgdenoisedbaboon}, bottom right-hand side) fits very well to the reference picture. 


\subsection{Grayscale image deblurring with TV regularization} \label{sec:TVdeblurring}
Following \cite{AcVog94,ChamLio97}, we minimize the functional
  \[
    T_\alpha(u)= \frac{1}{2} \int_\Omega ((Ku)(x)-u_0(x))^2\,dx + \alpha TV(u),
  \]
  where $TV(u)$ is defined as in \eqref{TVfunctgraycont}.
   Under the smoothness assumption $u \in W^{1,1}$, this leads to the parabolic gradient system
  \[
  u_t = \alpha \nabla \cdot \left[ \frac{\nabla u}{|\nabla u|_\beta} \right]
    -K^*(Ku-u_0),\qquad \mbox{with} \qquad 
    \left. \frac{\partial u}{\partial {\bf n}} \right|_{\partial \Omega} = 0.
  \] 
  The same procedure gives for the discretized functional
  \begin{equation} \label{deblVdisc}
    V_\alpha(u)=\frac{1}{2} \Delta x \Delta y \sum_{i=1}^{Nx} \sum_{j=1}^{Ny} \left(K u_{i,j} - (u_0)_{i,j},\right)^2+ \alpha J(u),
  \end{equation}
  where $J$ is given as before in \eqref{TVfunctgray}. We blur the image by convolution with the symmetric kernel $K$, given below as point spread function (PSF), which corresponds to the resulting image of a single bright pixel under the blurring transformation: 
\begin{equation} \label{PSF}
    K=
    \frac{1}{49}
    \left[
       \begin{array}{ccccccc}
          1 & 1 & 1 & 1 & 1 & 1 & 1 \\
          1 & 1 & 1 & 1 & 1 & 1 & 1 \\
          1 & 1 & 1 & 1 & 1 & 1 & 1 \\
          1 & 1 & 1 & 1 & 1 & 1 & 1 \\
          1 & 1 & 1 & 1 & 1 & 1 & 1 \\
          1 & 1 & 1 & 1 & 1 & 1 & 1 \\
          1 & 1 & 1 & 1 & 1 & 1 & 1        
       \end{array}
    \right]\,.
\end{equation}
In order to ensure Neumann boundary conditions, the original image $u$ is embedded in an image with four times the size of the original image, by reflecting the original image over the right-hand side and top boundaries. Then a two-dimensional convolution with the PSF in \eqref{PSF} is computed via the Fast Fourier Transform (FFT) and the resulting image of the correct size extracted. Details can be found in the nice introduction \cite{HNLdeblur06} to deblurring images. 
The original and the blurred image can be seen at the top of Figure~\ref{functblurbigsteppic}. In all experiments we used $\alpha=0.05$ and $\beta=0.001$. In Figure~\ref{functblurbigstep}, we first compare the value of the functional for subsequent steps of the Euler method with step size $\tau=0.19$ (cyan dotted line) to the value of the functional for subsequent steps of the midpoint discrete gradient method \eqref{eq:DGgonzalez} with step size $\tau=2.5$ (blue solid line). While the explicit Euler method does not produce a useful result, the image that corresponds to the computed equilibrium by the discrete gradient method is shown in Figure~\ref{functblurbigsteppic}, the right-hand side picture in the second row. For the solution of the implicit equation in the discrete gradient method, we have used the Newton method in the inner iteration with the exact Jacobian. The resulting linear system is solved by the Conjugate Gradient (CG) method (cf. \cite{HeStiCG}). 
\begin{figure}
\begin{center}
\begin{tikzpicture}
\begin{semilogxaxis}[width=15cm,height=7cm,xmin=0.75,xmax=11000,ymin=500,ymax=550, xtick={1,10,100, 1000, 10000}, ytick={510,520,530,540}, xticklabels={$1$, $10$, $100$, $1000$, $10000$},  yticklabels={$510$, $520$,$530$,$540$}, title={}, xlabel=\empty, ylabel=\empty]
\addplot[thick,color=red, dashed] coordinates {
( 1, 534.4168 )
( 2, 520.1117 )
( 3, 515.0396 )
( 4, 512.5979 )
( 5, 511.2594 )
( 6, 510.4553 )
( 7, 509.9365 )
( 8, 509.5838 )
( 9, 509.3354 )
( 10, 509.1575 )
( 20, 509.9863 )
( 30, 520.0123 )
( 40, 527.0358 )
( 50, 529.2291 )
( 60, 529.9959 )
( 70, 530.3589 )
( 80, 530.5833 )
( 90, 530.7405 )
( 100, 530.8563 )
( 200, 531.4839 )
( 300, 531.6957 )
( 400, 531.8749 )
( 500, 532.0213 )
( 600, 532.0921 )
( 700, 532.1352 )
( 800, 532.181 )
( 900, 532.2379 )
( 1000, 532.264 )
( 2000, 532.3961 )
( 3000, 532.4537 )
( 4000, 532.4636 )
( 5000, 532.4725 )
( 6000, 532.4926 )
( 7000, 532.497 )
( 8000, 532.497 )
( 9000, 532.4971 )
( 10000, 532.4971 )
};
\addplot[thick,color=red, dashed] coordinates {
( 1, 534.4168 )
( 2, 520.3609 )
( 3, 515.2505 )
( 4, 512.7592 )
( 5, 511.3806 )
( 6, 510.547 )
( 7, 510.0064 )
( 8, 509.6363 )
( 9, 509.3728 )
( 10, 509.1799 )
( 20, 508.8218 )
( 30, 512.2619 )
( 40, 518.3166 )
( 50, 521.5199 )
( 60, 522.7868 )
( 70, 523.3454 )
( 80, 523.6493 )
( 90, 523.8401 )
( 100, 523.967 )
( 200, 524.5044 )
( 300, 524.672 )
( 400, 524.8069 )
( 500, 524.892 )
( 600, 524.9607 )
( 700, 524.987 )
( 800, 525.0214 )
( 900, 525.0471 )
( 1000, 525.0801 )
( 2000, 525.2049 )
( 3000, 525.2374 )
( 4000, 525.2706 )
( 5000, 525.2749 )
( 6000, 525.2761 )
( 7000, 525.2773 )
( 8000, 525.2781 )
( 9000, 525.2781 )
( 10000, 525.2781 )
};
\addplot[thick,color=red, dashed] coordinates {
( 1, 534.4168 )
( 2, 520.6182 )
( 3, 515.4711 )
( 4, 512.9304 )
( 5, 511.5109 )
( 6, 510.6475 )
( 7, 510.085 )
( 8, 509.6985 )
( 9, 509.4219 )
( 10, 509.2179 )
( 20, 508.5917 )
( 30, 508.9913 )
( 40, 511.233 )
( 50, 514.0682 )
( 60, 515.8566 )
( 70, 516.7978 )
( 80, 517.3106 )
( 90, 517.6156 )
( 100, 517.81 )
( 200, 518.3421 )
( 300, 518.4909 )
( 400, 518.579 )
( 500, 518.6561 )
( 600, 518.6807 )
( 700, 518.7051 )
( 800, 518.7279 )
( 900, 518.7555 )
( 1000, 518.7689 )
( 2000, 518.8408 )
( 3000, 518.8743 )
( 4000, 518.8781 )
( 5000, 518.8822 )
( 6000, 518.8849 )
( 7000, 518.909 )
( 8000, 518.9089 )
( 9000, 518.9108 )
( 10000, 518.9108 )
};
\addplot[thick,color=cyan, dotted] coordinates {
( 1, 534.4168 )
( 2, 519.8706 )
( 3, 514.8381 )
( 4, 512.4461 )
( 5, 511.1475 )
( 6, 510.3731 )
( 7, 509.8774 )
( 8, 509.5451 )
( 9, 509.318 )
( 10, 509.1662 )
( 20, 513.694 )
( 30, 530.0458 )
( 40, 535.615 )
( 50, 536.9858 )
( 60, 537.5079 )
( 70, 537.8126 )
( 80, 538.0195 )
( 90, 538.1796 )
( 100, 538.3063 )
( 200, 539.0465 )
( 300, 539.3467 )
( 400, 539.5461 )
( 500, 539.6892 )
( 600, 539.7961 )
( 700, 539.8672 )
( 800, 539.9257 )
( 900, 539.9768 )
( 1000, 540.0415 )
( 2000, 540.2137 )
( 3000, 540.2951 )
( 4000, 540.3384 )
( 5000, 540.3492 )
( 6000, 540.3498 )
( 7000, 540.3497 )
( 8000, 540.3497 )
( 9000, 540.3496 )
( 10000, 540.3496 )
};
\addplot[thick,color=green, dashdotted] coordinates {
( 1, 534.4168 )
( 2, 525.421 )
( 3, 520.2854 )
( 4, 517.0997 )
( 5, 514.9942 )
( 6, 513.5363 )
( 7, 512.4911 )
( 8, 511.7201 )
( 9, 511.1376 )
( 10, 510.6878 )
( 20, 508.9839 )
( 30, 508.6204 )
( 40, 508.5051 )
( 50, 508.4595 )
( 60, 508.4381 )
( 70, 508.4266 )
( 80, 508.4199 )
( 90, 508.4157 )
( 100, 508.4131 )
( 200, 508.4072 )
( 300, 508.4069 )
( 400, 508.4068 )
( 500, 508.4068 )
( 600, 508.4068 )
( 700, 508.4068 )
( 800, 508.4068 )
( 900, 508.4068 )
( 1000, 508.4068 )
( 10000, 508.4068 )
};
\addplot[thick,color=green, dashdotted] coordinates {
( 1, 534.4168 )
( 2, 533.3926 )
( 3, 532.4197 )
( 4, 531.4952 )
( 5, 530.6163 )
( 6, 529.7804 )
( 7, 528.9851 )
( 8, 528.2282 )
( 9, 527.5075 )
( 10, 526.8211 )
( 20, 521.4707 )
( 30, 518.0239 )
( 40, 515.7067 )
( 50, 514.0896 )
( 60, 512.9253 )
( 70, 512.0644 )
( 80, 511.4133 )
( 90, 510.9109 )
( 100, 510.5164 )
( 200, 508.9624 )
( 300, 508.6164 )
( 400, 508.5042 )
( 500, 508.4592 )
( 600, 508.4379 )
( 700, 508.4265 )
( 800, 508.4199 )
( 900, 508.4157 )
( 1000, 508.4131 )
( 2000, 508.4072 )
( 3000, 508.4069 )
( 4000, 508.4068 )
( 10000, 508.4068 )
};
\addplot[thick,color=green, dashdotted] coordinates {
( 1, 534.4168 )
( 2, 534.3132 )
( 3, 534.2101 )
( 4, 534.1075 )
( 5, 534.0055 )
( 6, 533.904 )
( 7, 533.803 )
( 8, 533.7025 )
( 9, 533.6025 )
( 10, 533.503 )
( 20, 532.5352 )
( 30, 531.6147 )
( 40, 530.7389 )
( 50, 529.9054 )
( 60, 529.1119 )
( 70, 528.3561 )
( 80, 527.6359 )
( 90, 526.9496 )
( 100, 526.2951 )
( 200, 521.1694 )
( 300, 517.8408 )
( 400, 515.5896 )
( 500, 514.0116 )
( 600, 512.8714 )
( 700, 512.0262 )
( 800, 511.3854 )
( 900, 510.8901 )
( 1000, 510.5005 )
( 2000, 508.9604 )
( 3000, 508.616 )
( 4000, 508.5041 )
( 5000, 508.4592 )
( 6000, 508.4379 )
( 7000, 508.4265 )
( 8000, 508.4199 )
( 9000, 508.4157 )
( 10000, 508.4131 )
};
\addplot[thick,color=blue] coordinates {
( 1, 534.4168 )
( 2, 514.2426 )
( 3, 510.5953 )
( 4, 509.3868 )
( 5, 508.8855 )
( 6, 508.6533 )
( 7, 508.5381 )
( 8, 508.4786 )
( 9, 508.4468 )
( 10, 508.4295 )
( 20, 508.407 )
( 30, 508.4068 )
( 1000, 508.4069 )
( 10000, 508.4069 )
};
\end{semilogxaxis}
\end{tikzpicture} 
\end{center}
\vspace{-1cm}
\caption{TV deblurring functional Baboon, explicit Euler method with step size $\tau=0.19$ (cyan dotted line), explicit Euler with $\tau=0.185, 0.18, 0.175$ (red dashed lines, top to bottom), explicit Euler with $\tau=0.1, 0.01, 0.001$ (green dash-dotted lines, left to right), discrete gradient method (blue solid line) with step size $\tau=2.5$ (in all experiments: $\alpha=0.05$, $\beta=0.001$).}
\label{functblurbigstep}
\end{figure}

\begin{figure}  
  \begin{center}
  \includegraphics[width=6cm]{baboon.pdf}
  \hspace{2cm}
  \includegraphics[width=6cm]{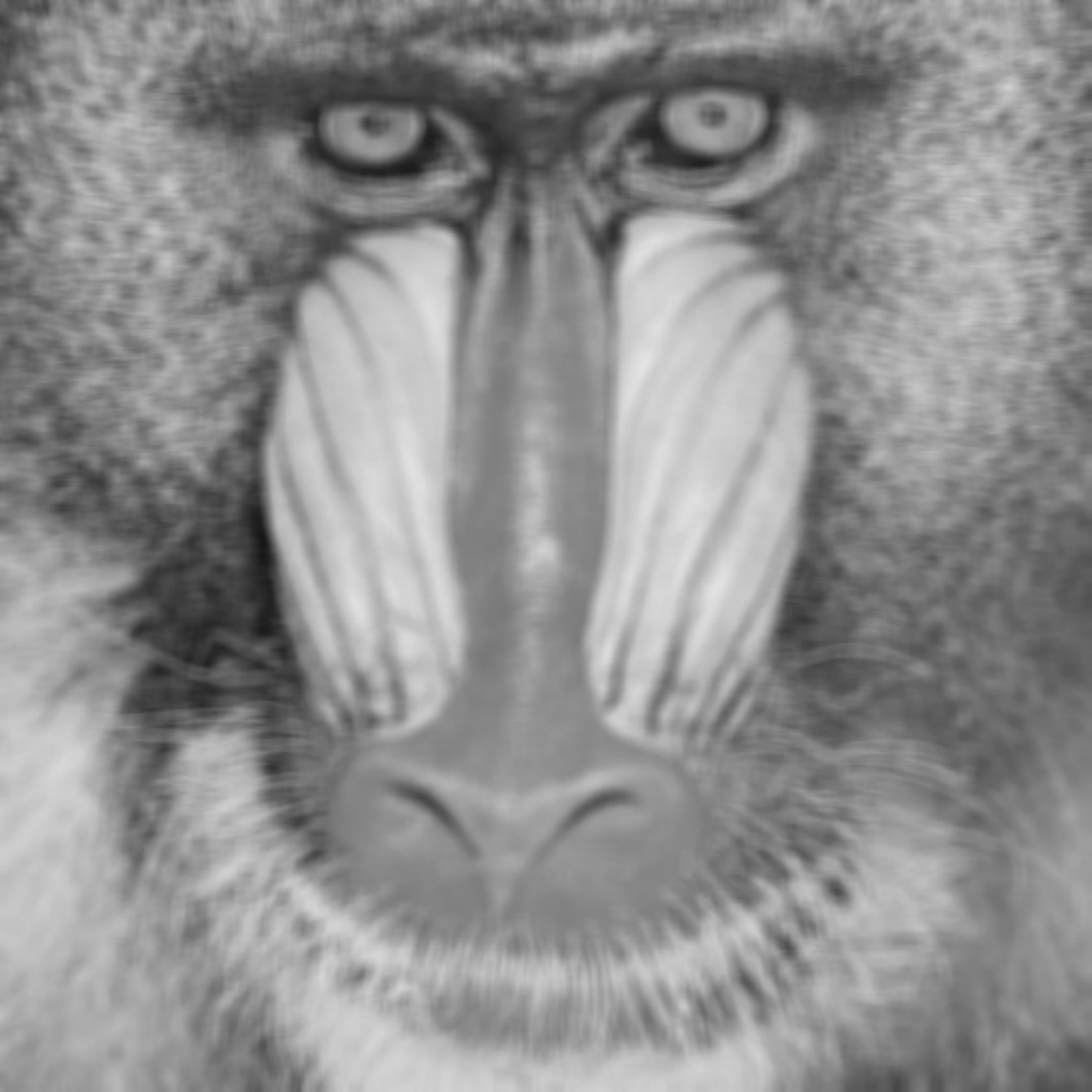}
  \end{center}
  
  \vspace{1cm}
  
  \begin{center}
  \includegraphics[width=6cm]{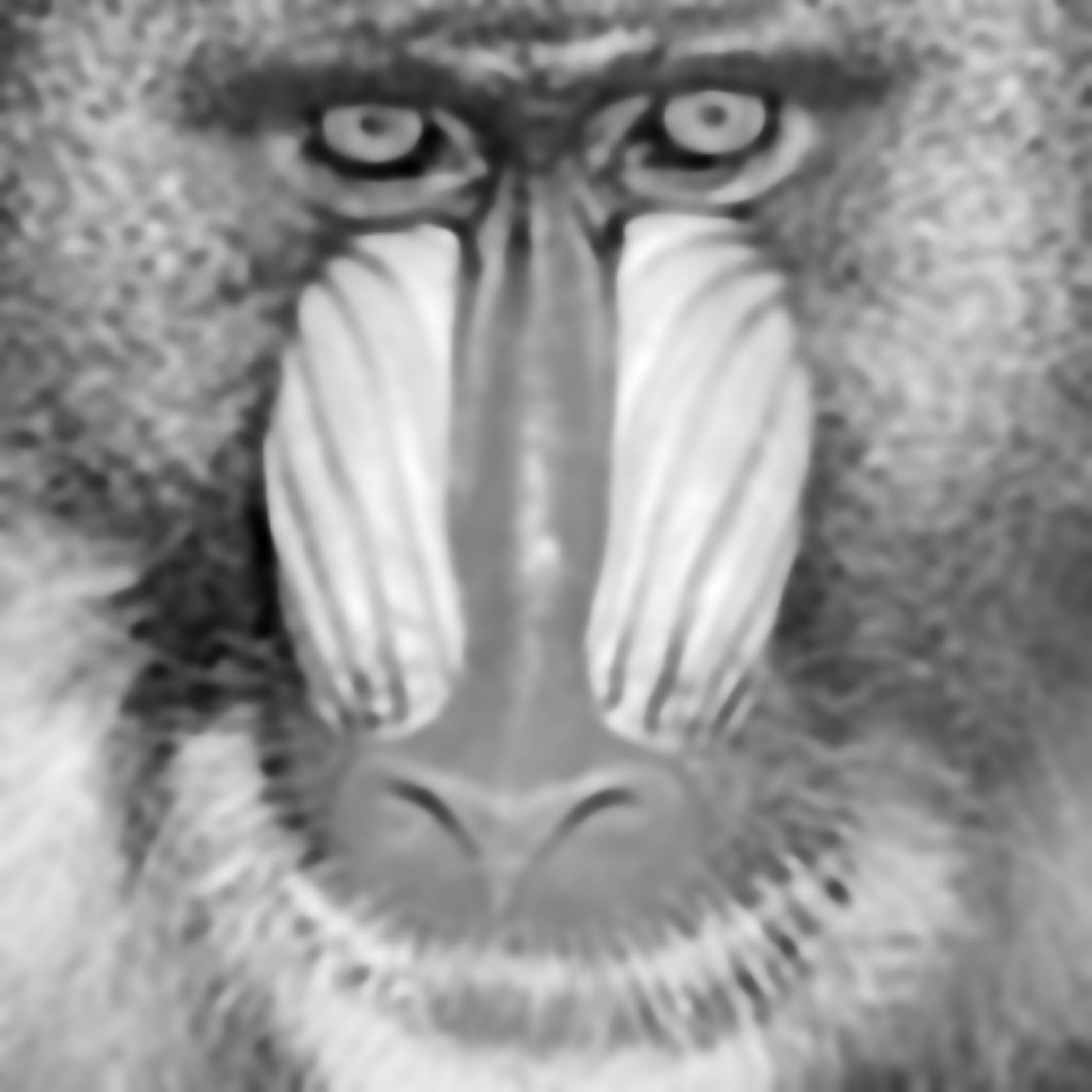}
  \hspace{2cm}
  \includegraphics[width=6cm]{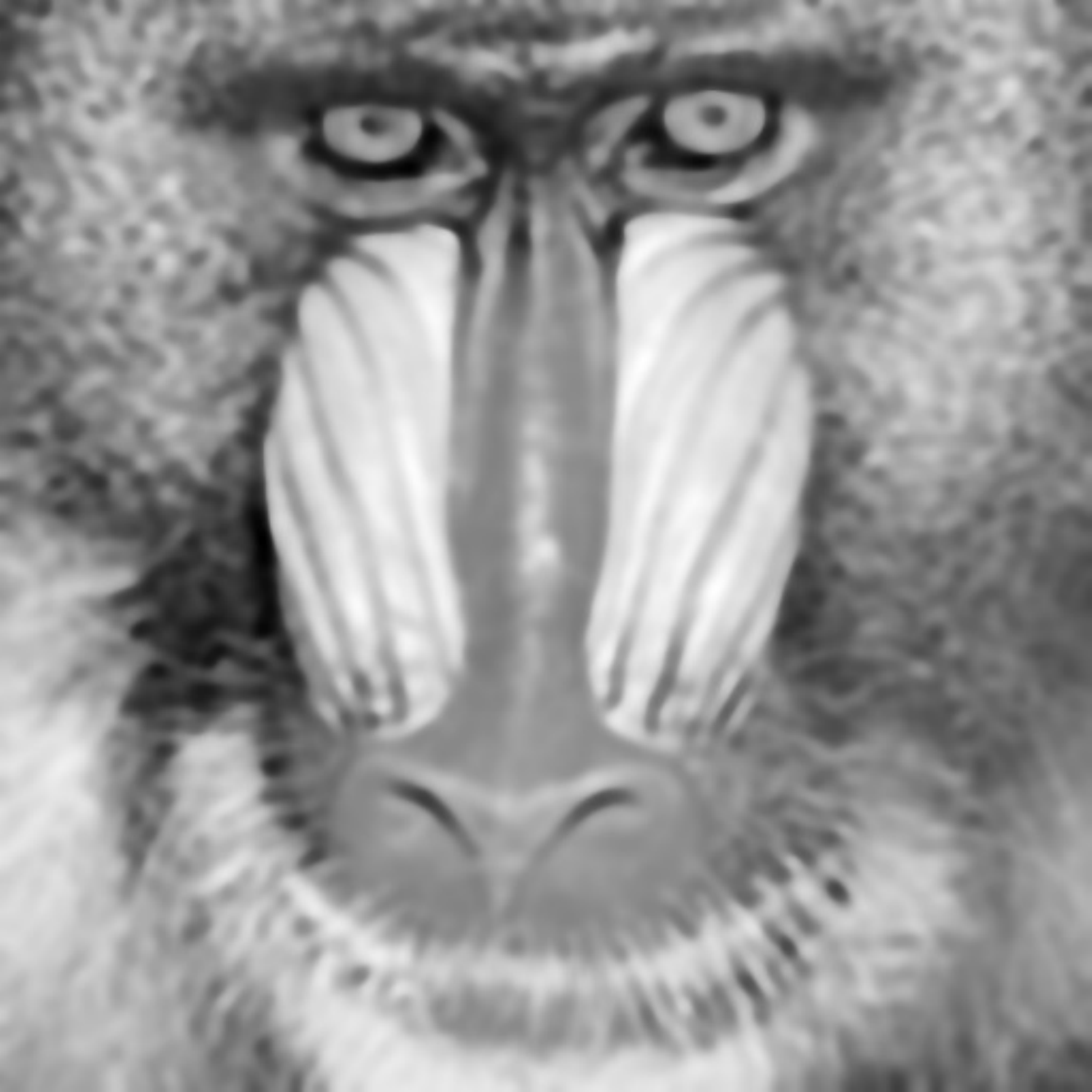}
  \end{center}  
  \caption{Original picture (top left), blurred picture (top right), TV-denoised reference picture with the explicit Euler method, step-size $\tau=0.01$ after $10000$ steps (bottom left), 
 discrete gradient denoised baboon with step-size $\tau=2.25$ after $10$ steps (bottom right).} \label{functblurbigsteppic}  
\end{figure}

We repeat the experiment with the step sizes $\tau=0.185, 0.18, 0.175$ (red dashed lines, top to bottom). The explicit Euler method converges for these step sizes.  The values of the functional of the explicit Euler method can also be seen in Figure~\ref{functblurbigstep}. The explicit Euler method obviously converges to an incorrect equilibrium  value of the function for these step sizes. The explicit Euler method with the step sizes $\tau=0.1, 0.01, 0.001$ (green dash-dotted lines, left to right) converge to the same equilibrium value of the functional as the discrete gradient method does for the step size $\tau=2.5$ (blue, solid line). But after $10000$ Euler steps with $\tau=0.001$, the reached value of the functional is $508.4131$ and still larger than $508.4069$, the value reached by the discrete gradient method with step size $\tau=2.25$ in step $30$ (blue, solid line).
An experiment with the discrete gradient method shows that the equilibrium picture computed by the discrete gradient method is the same for the different step sizes $\tau=2.5, 0.1, 0.185, 0.18, 0.175, 0.01, 0.001$. The equilibrium picture of the discrete gradient method after $10$ steps can be seen in Figure~\ref{functblurbigsteppic} in the bottom row on the right-hand side. 
The picture with step-size $\tau=0.01$ after $10000$ steps has been used as a reference for the correct equilibrium and can be seen at the bottom of Figure~\ref{functblurbigsteppic} on the left-hand side. The general observation is that the minimum value of the functional (and therefore the equilibrium image) found by the explicit Euler method clearly depends on the step size. This is unpleasant with respect to a reliable computation of minimizers (corresponding to smoothed images) for given smoothing parameters $\alpha$. On the contrary, the discrete gradient method seems to find the correct minimum value of the functional for a broad range of step sizes, due to the preservation of dissipativity.

\subsection{Grayscale image inpainting with TV regularization and adaptive step size} 
\label{sec:TVinpainting}

We consider the functional 
 \[
    T_\alpha(u)= \frac{1}{2} \int_{\Omega\setminus D} (u(x)-u_0(x))^2\,dx + \alpha TV(u),
  \]
where $TV$ is defined as in \eqref{TVfunctgraycont} and $D$ is a subset of $\Omega$ in which no information about $u_0$ is available. This can be the case because either the image is damaged in $D$ or the original scene in the image is occluded by something else in $D$. The task is to recover the original image in $D$ by minimizing the functional above. This is called image inpainting. The associated discretized functional is given by

\begin{equation} \label{inpaintVdisc}
    V_\alpha(u)=\frac{1}{2} \Delta x \Delta y \sum_{i=1}^{Nx} \sum_{j=1}^{Ny} \mathbb P\left(u_{i,j} - (u_0)_{i,j}\right)^2+ \alpha J(u),
  \end{equation}
where $J$ is given as before in \eqref{TVfunctgray} and $\mathbb P$ is a projection of the data fitting term onto the set of indices $(i,j)$ that our outside the inpainting domain $D$. 

In the following example, we compare the minimization of the TV-inpainting functional using the discrete gradient method with the Itoh \& Abe gradient \eqref{itoabe:discretegrad} and adaptive step size (as described below) with its minimization by the lagged-diffusivity method \cite{AcVog94,chan1999convergence}. In the latter, a minimizer of the TV-inpainting functional is characterized by a solution of the corresponding Euler-Lagrange equation and the following fixed-point iteration is performed 
\begin{equation}\label{laggeddiff}
0 = \mathrm{div}\left(\frac{\nabla u_{n+1}}{\sqrt{|\nabla u_n|^2+\epsilon}}\right) + \mathbbm{1}_{\Omega\setminus D}(u_0-u_{n+1}),
\end{equation}
evaluating the nonlinearity in the previous time-step only, and where $\mathbbm 1_{\Omega\setminus D}$ is the characteristic function of the set $\Omega\setminus D$. Here $\nabla u$ is discretized as in \eqref{TVfunctgray} with backward finite differences and its negative adjoint the divergence $\mathrm{div}$ by forward finite differences.

For the Itoh \& Abe discrete gradient approach we use a simple time step adaptation. In every iteration we compute two trial steps with time steps $\tau$ and $2\tau$ and choose the one that decreases $V_\alpha$ most. If the chosen solution corresponds to the time step $\tau$ then we halve the time step for the next step, otherwise we double it.

The example in Figure \ref{inpaintsurfDG} is a gray scale image of size $219\times 292$. The inpainting task is to remove the superimposed text from the image and replace it by the TV-interpolation of the surrounding gray values. Figure \ref{inpaintsurfenergies} reports the energy decrease for the Itoh-Abe discrete gradient method compared to the lagged-diffusivity iteration, and the evolution of the step sizes which were adaptively chosen throughout the discrete gradient iterations. In this experiment $\alpha=0.0001$ and $\beta=0.01$. 

Note that the equations to be solved in an Itoh \& Abe update for $V_\alpha$ under the Euclidean inner product as considered here uncouple to scalar equations. Yet, in Figure \ref{inpaintsurfenergies} it still appears to choose good descent directions even for large time steps. As one can also observe in Figure \ref{inpaintsurfenergies} the energy decrease with lagged-diffusivity is monotonic and faster than under the discrete gradient iteration. This qualitative behavior is representative for the application of lagged-diffusivity to convex functionals $V_\alpha$. Monotonicity, however, breaks in the case of non-convex functionals for which we will see in Subsection \ref{sec:TVnonconvex} the discrete gradient method still preserves monotonic decrease. 

\begin{figure}[h!] 
  \begin{center}
  \includegraphics[width=0.48\textwidth]{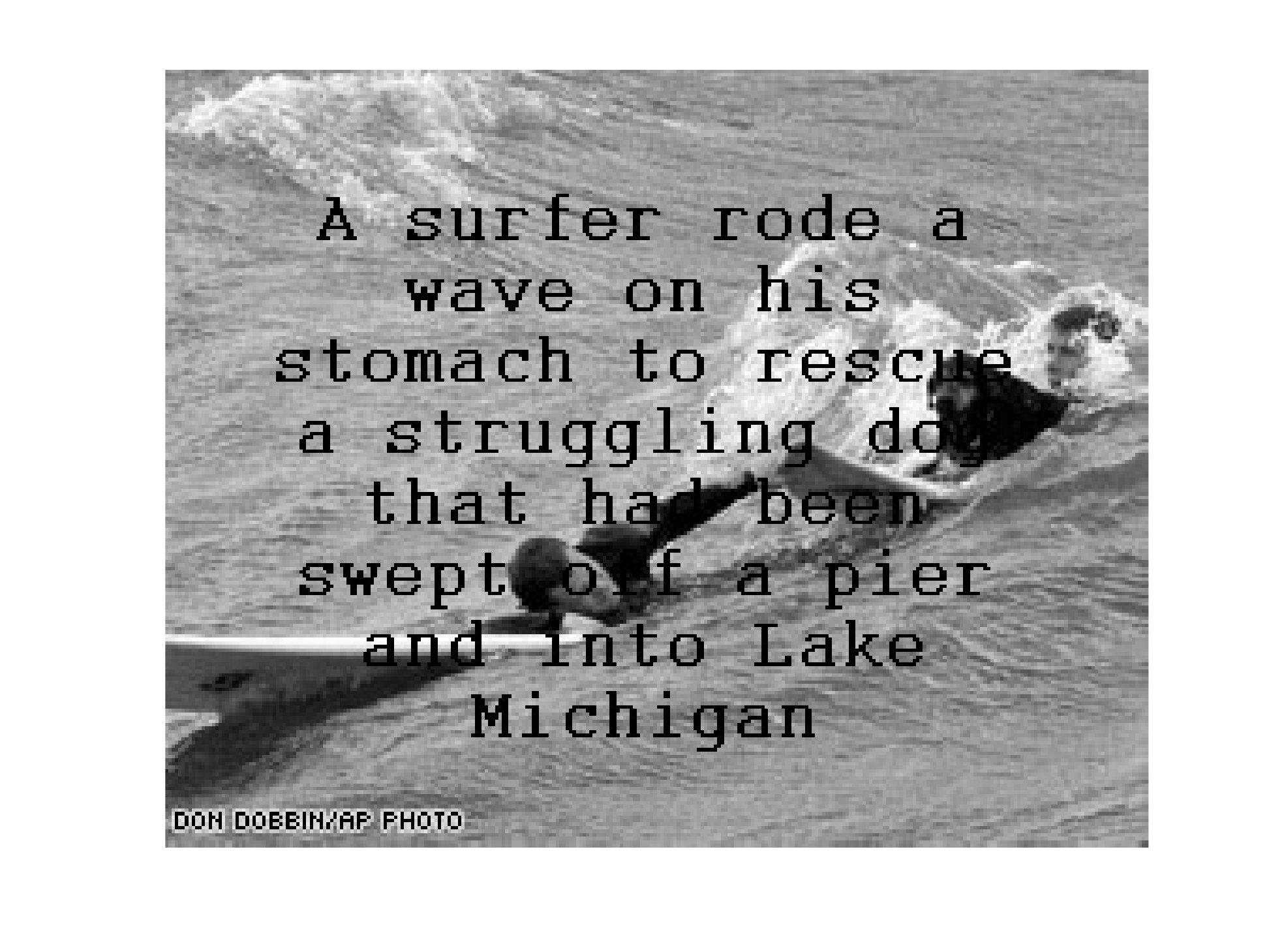}
  \includegraphics[width=0.48\textwidth]{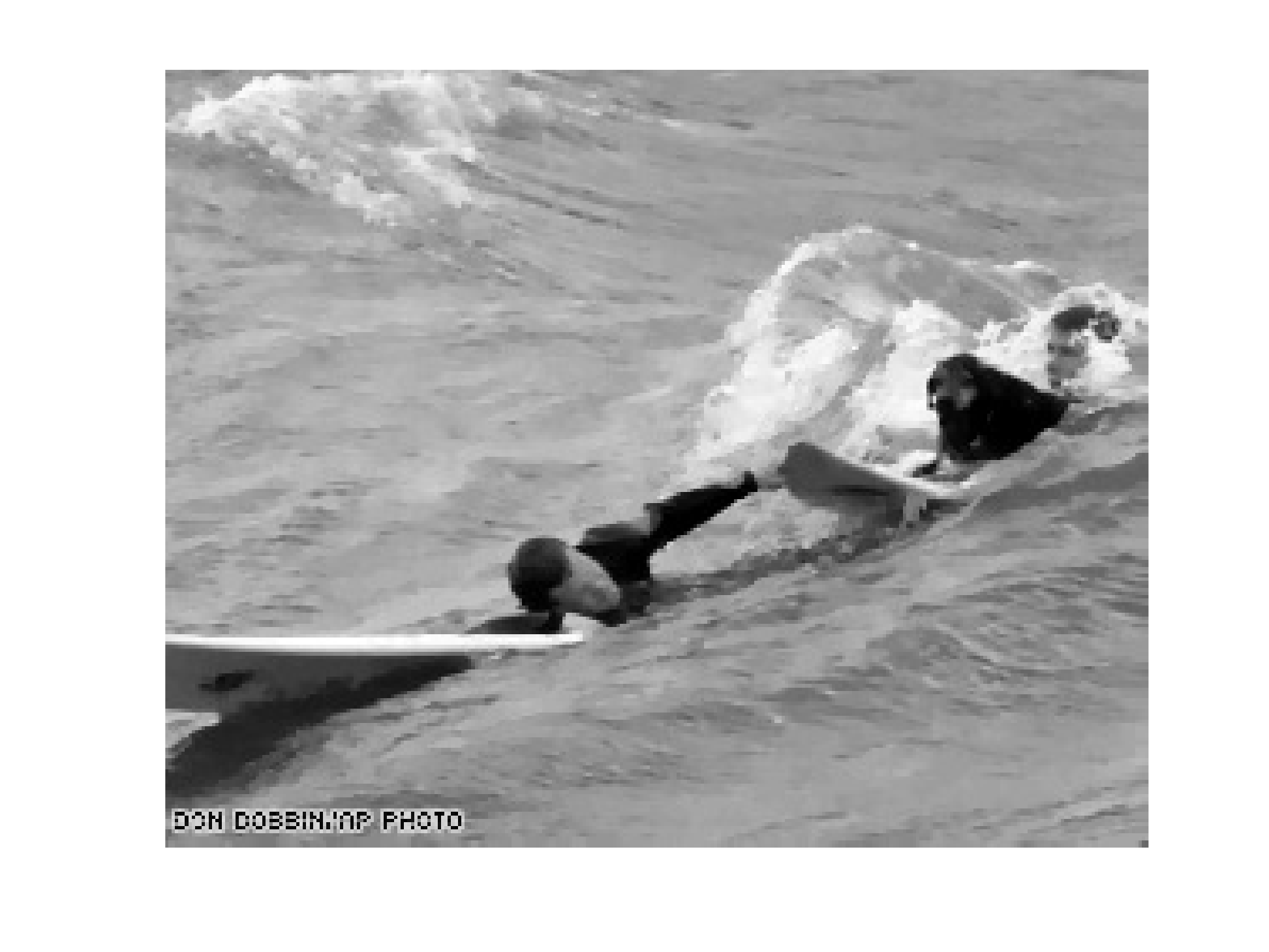}
  \end{center}
  \caption{Original image with superimposed writing and TV-inpainted image with DG method and adaptively chosen $\tau$.}
  \label{inpaintsurfDG}
\end{figure}

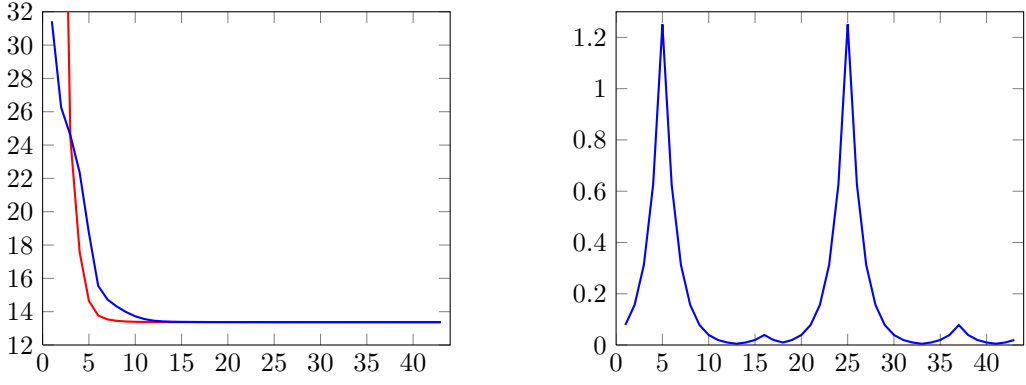
\begin{figure}[h!]
\begin{center}
\hspace{-1cm}
 {\small
   \begin{tikzpicture}
\begin{axis}[width=7cm,height=6cm,xmin=0.0,xmax=44,ymin=12,ymax=32, xtick={0, 5, 10, 15, 20, 25, 30, 35, 40, 45, 50}, ytick={12, 14, 16, 18, 20, 22, 24, 26, 28, 30, 32}, xticklabels={$0$, $5$, $10$, $15$, $20$, $25$, $30$, $35$, $40$, $45$, $50$},  yticklabels={$12$, $14$, $16$, $18$, $20$, $22$, $24$, $26$, $28$, $30$, $32$}, title={}, xlabel=\empty, ylabel=\empty]
\addplot[thick,color=red] coordinates {
( 1, 1750.4763 )
( 2, 55.3621 )
( 3, 24.3951 )
( 4, 17.5667 )
( 5, 14.6451 )
( 6, 13.7708 )
( 7, 13.5385 )
( 8, 13.4496 )
( 9, 13.4103 )
( 10, 13.3915 )
( 11, 13.3821 )
( 12, 13.377 )
( 13, 13.3741 )
( 14, 13.3724 )
( 15, 13.3714 )
( 16, 13.3707 )
( 17, 13.3703 )
( 18, 13.37 )
( 19, 13.3697 )
( 20, 13.3696 )
( 21, 13.3695 )
( 22, 13.3694 )
( 23, 13.3694 )
( 24, 13.3693 )
( 25, 13.3693 )
( 26, 13.3692 )
( 27, 13.3692 )
( 28, 13.3692 )
( 29, 13.3692 )
( 30, 13.3692 )
( 31, 13.3692 )
( 32, 13.3692 )
( 33, 13.3692 )
( 34, 13.3691 )
( 35, 13.3691 )
( 36, 13.3691 )
( 37, 13.3691 )
( 38, 13.3691 )
( 39, 13.3691 )
( 40, 13.3691 )
( 41, 13.3691 )
( 42, 13.3691 )
( 43, 13.3691 )
};
\addplot[thick,color=blue] coordinates {
( 1, 31.4287 )
( 2, 26.259 )
( 3, 24.5786 )
( 4, 22.3446 )
( 5, 18.7171 )
( 6, 15.553 )
( 7, 14.7289 )
( 8, 14.3131 )
( 9, 13.9808 )
( 10, 13.7244 )
( 11, 13.555 )
( 12, 13.4597 )
( 13, 13.415 )
( 14, 13.3964 )
( 15, 13.3888 )
( 16, 13.3842 )
( 17, 13.3809 )
( 18, 13.3786 )
( 19, 13.3768 )
( 20, 13.3756 )
( 21, 13.3749 )
( 22, 13.3743 )
( 23, 13.3737 )
( 24, 13.373 )
( 25, 13.3722 )
( 26, 13.3715 )
( 27, 13.3709 )
( 28, 13.3705 )
( 29, 13.3703 )
( 30, 13.3701 )
( 31, 13.37 )
( 32, 13.3698 )
( 33, 13.3697 )
( 34, 13.3696 )
( 35, 13.3695 )
( 36, 13.3694 )
( 37, 13.3694 )
( 38, 13.3693 )
( 39, 13.3693 )
( 40, 13.3693 )
( 41, 13.3692 )
( 43, 13.3692 )
};
\end{axis}
\end{tikzpicture} 
   \hspace{0.5cm}
   \begin{tikzpicture}
\begin{axis}[width=7cm,height=6cm,xmin=0.0,xmax=44, ymin=0,ymax=0.0013, xtick={0, 5, 10, 15, 20, 25, 30, 35, 40, 45, 50}, ytick={0, 0.0002, 0.0004, 0.0006, 0.0008, 0.001, 0.0012, 0.0014, 0.0016}, xticklabels={$0$, $5$, $10$, $15$, $20$, $25$, $30$, $35$, $40$, $45$, $50$},  yticklabels={$0$, $0.2$, $0.4$, $0.6$, $0.8$, $1$, $1.2$, $1.4$,$1.6$}, title={}, xlabel=\empty, ylabel=\empty]
\addplot[thick,color=blue] coordinates {
( 1, 7.8189e-05 )
( 2, 0.00015638 )
( 3, 0.00031275 )
( 4, 0.00062551 )
( 5, 0.001251 )
( 6, 0.00062551 )
( 7, 0.00031275 )
( 8, 0.00015638 )
( 9, 7.8189e-05 )
( 10, 3.9094e-05 )
( 11, 1.9547e-05 )
( 12, 9.7736e-06 )
( 13, 4.8868e-06 )
( 14, 9.7736e-06 )
( 15, 1.9547e-05 )
( 16, 3.9094e-05 )
( 17, 1.9547e-05 )
( 18, 9.7736e-06 )
( 19, 1.9547e-05 )
( 20, 3.9094e-05 )
( 21, 7.8189e-05 )
( 22, 0.00015638 )
( 23, 0.00031275 )
( 24, 0.00062551 )
( 25, 0.001251 )
( 26, 0.00062551 )
( 27, 0.00031275 )
( 28, 0.00015638 )
( 29, 7.8189e-05 )
( 30, 3.9094e-05 )
( 31, 1.9547e-05 )
( 32, 9.7736e-06 )
( 33, 4.8868e-06 )
( 34, 9.7736e-06 )
( 35, 1.9547e-05 )
( 36, 3.9094e-05 )
( 37, 7.8189e-05 )
( 38, 3.9094e-05 )
( 39, 1.9547e-05 )
( 40, 9.7736e-06 )
( 41, 4.8868e-06 )
( 42, 9.7736e-06 )
( 42, 9.7736e-06 )
( 43, 1.9547e-05 )
};
\end{axis}
\end{tikzpicture} 
 }
\end{center}
  \caption{Left: Energy decrease for TV-inpainting result in Figure \ref{inpaintsurfDG} with the Itoh-Abe DG method and adaptive step size (blue) and lagged diffusivity with $\tau=0.1$ (red). Right: Adaptive step sizes for the Itoh-Abe discrete gradient.}
  \label{inpaintsurfenergies}
\end{figure}

\subsection{Multichannel image denoising with TV regularization} \label{sec:TVmultichannel}
We check the fitness of the discrete gradient method for an application in the real world by an experiment with the discrete gradient method applied to an image processing task in macro photography. We use the multichannel model as described in \cite{ChanShenbook05} and first introduced in \cite{BloChan98}, which uses the $TV$-functional:
\[
TV_2[u]=\left( \sum_{i=1}^p \left(TV[u_i]\right)^2 \right)^{1\slash 2} 
       =\left( \sum_{i=1}^p \left( \int_\Omega |Du_i|\, dx\right)^2 \right)^{1\slash 2}  
\]
where $TV(u_i)$ is defined as in \eqref{TVfunctgraycont} for $p$ channels $u_i$, $i=1,\ldots,p$, in the denoising functional
\[
T_\alpha(u)=\alpha TV_2[u] + \frac{1}{2} \int_\Omega \|u-u_0\|^2\, dx\,.
\]
 With the global constants
\[
  c_i[u]=\frac{TV[u_i]}{TV_2[u]} \geq 0, \qquad i=1,\ldots,p\,,
\]
the Euler-Lagrange equilibrium system reads
\[
-\alpha c_i[u] \nabla \cdot \left[ \frac{\nabla u_i}{|\nabla u_i|_\beta}\right]
+  (u_i-u_{0,i}) = 0, \quad \left. \frac{\partial u_i}{\partial {\bf n}} \right|_{\partial \Omega} = 0, \quad i=1,\ldots,p\,.
\]
Time-marching leads to the gradient system
\[
\frac{d}{dt} u_i= \alpha \cdot c_i[u] \nabla \cdot \left[ \frac{\nabla u_i}{|\nabla u_i|_\beta}\right]
-  (u_i-u_{0,i}) = 0, \quad \left. \frac{\partial u_i}{\partial {\bf n}} \right|_{\partial \Omega} = 0, \quad i=1,\ldots,p\,.
\]
The discretized system is just given by using the discretized TV-functionals in $\mbox{TV}_2$. The corresponding equations are then solved. The situation is analogous to the case of grayscale image denoising.
\begin{proposition} The functional $V_\alpha$ corresponding to the discretized multichannel TV denoising functional possesses a unique minimizer and the sequence generated by any discrete gradient method with step sizes $0 < c \leq \tau_n \leq M < \infty$ converges to the unique minimizer.
\end{proposition}
\begin{proof}
With the help of Lemma~\ref{Denoisegraylem}, one can check that the discretized functional is bounded from below, coercive, continuously differentiable and strictly convex. The statement then follows from Theorem~\ref{seqtostat}.
\end{proof}

Encouraged by the results for the smaller test images before, we apply the discrete gradient method with the midpoint discrete gradient to a real world denoising problem. The picture at the top of Figure~\ref{lice} is an original photography of some plant lice. The picture has been taken with a strong macro lens, the Canon MP-E 65mm macro lens, that exhibits an extremely low depth-of-field, ranging from 2.24mm at f/16 at 1x magnification, and a minimum of 0.048mm at f/2.8 at 5x magnification. As a camera, a Canon EOS 550D camera has been used, hand-held in full sunlight, with an exposure time of 1/250 and  f-stop number 14 at 3x magnification. The film speed has been set to ISO 6400, which was needed due to make an exposure time of 1/250 possible. The drawback of this approach to take macro photos without flash is that the high film speed produces a lot of noise due to the necessary amplification of the signal from the charge-coupled device (CCD) image sensor. This real-life noise can clearly be seen in the picture at the top of  Figure~\ref{lice} and in the picture detail on the left-hand side in Figure~\ref{licesmall}. The image size in width~$\times$~height is 5184~$\times$~3456. The overall denoising gradient system for an RGB picture therefore is of dimension $n=3~\times~5184~\times~3456=53747712$.
\begin{figure}   
  \begin{center}
    \includegraphics[width=6.5cm]{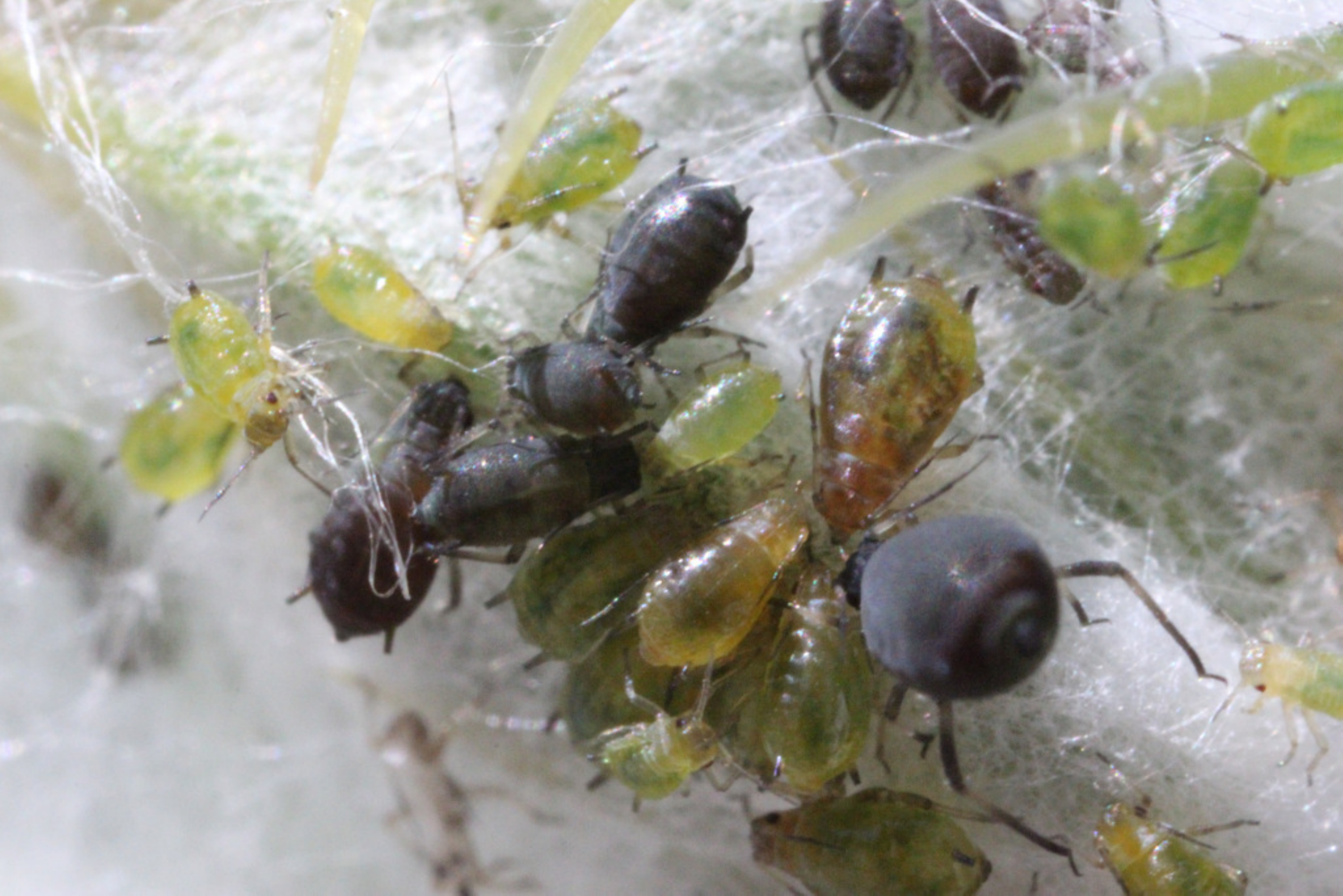}
    \hspace{0.5cm}
    \includegraphics[width=6.5cm]{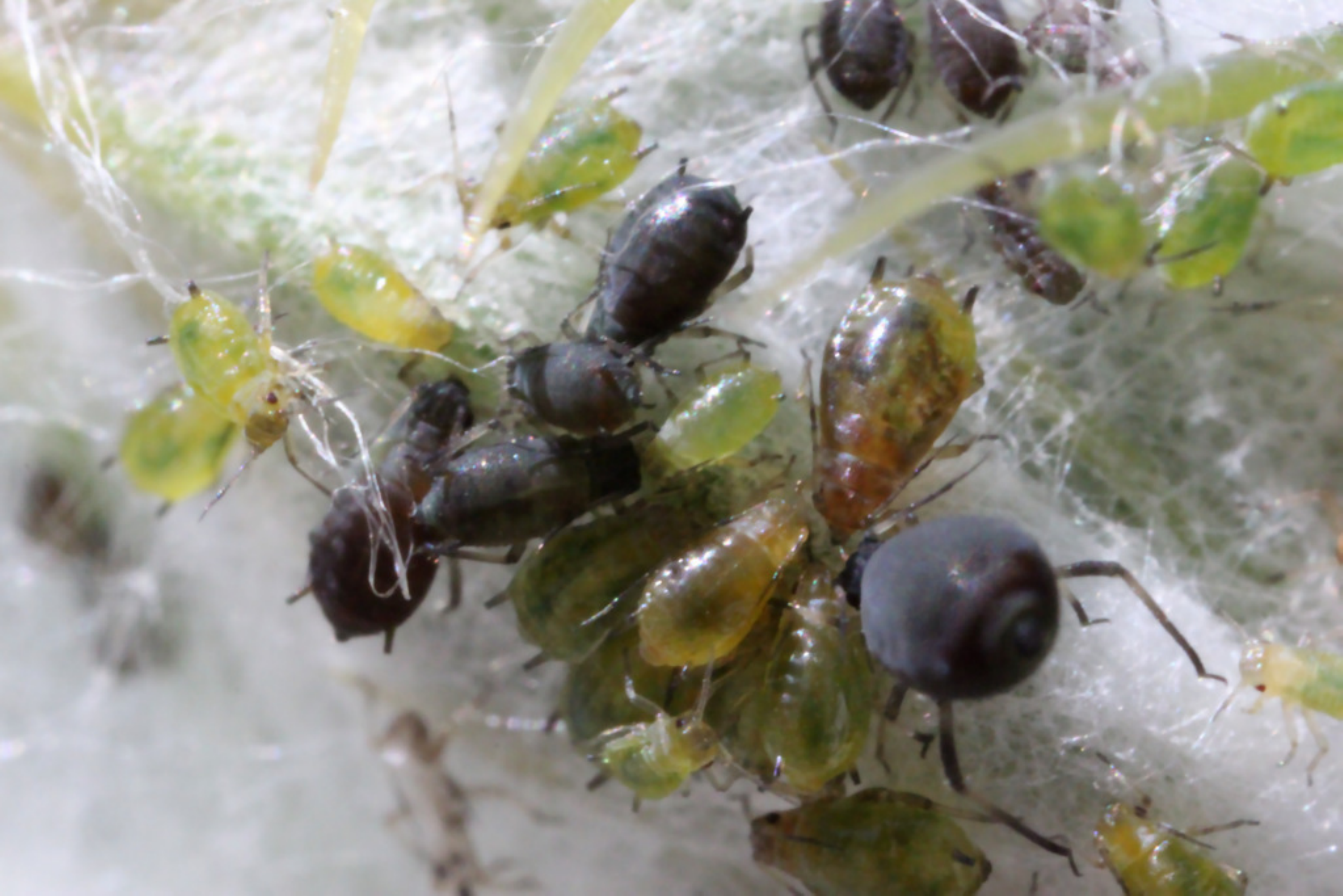}
  \end{center}
  \caption{Original (left) and denoised (right) lice picture, \copyright Volker Grimm} \label{lice}
\end{figure}

\begin{figure}                     
  \begin{center}
    \includegraphics[width=6.5cm]{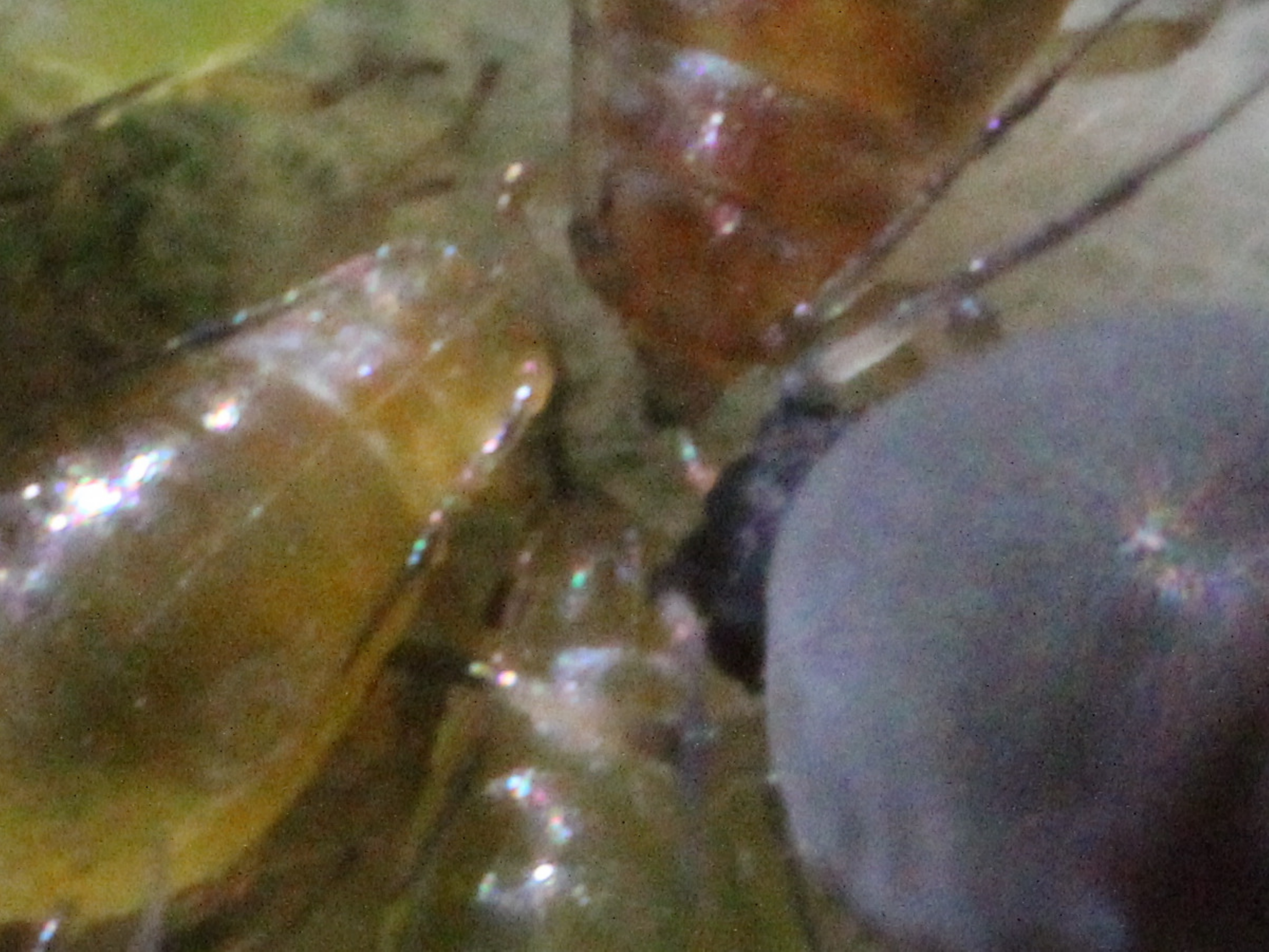}
    \hspace{0.5cm}
    \includegraphics[width=6.5cm]{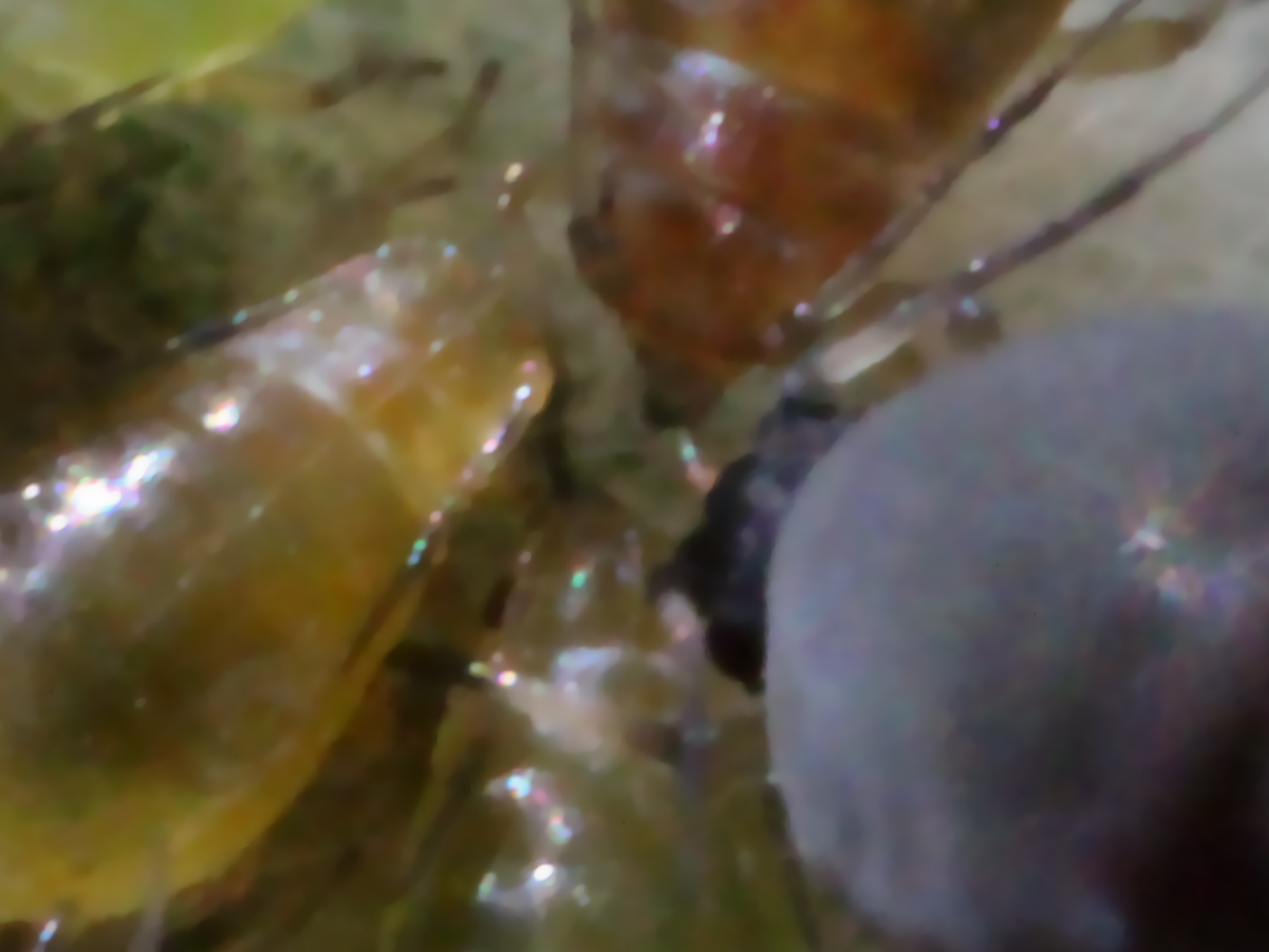}
  \end{center}
  \caption{Zoom of original image (left) and denoised image (right), \copyright Volker Grimm} \label{licesmall}
\end{figure}
Despite the size of the system, the discrete gradient method preserves the dissipativity and converges in $10$ steps with step size $\tau=0.2$ to the equilibrium picture.
The image has been rescaled to pixel size $\Delta x=\Delta y=1$, the image data has been in the interval $[0,255]$, and the constants have been  chosen as $\beta=1$ and $\alpha=100$.  
In Figure~\ref{lice} and in the detail in Figure~\ref{licesmall}, one can see that the discrete gradient method successfully removes the noise from the original lice photography.  

\subsection{A non-convex example: grayscale image denoising with TV$^p$ regularization, $0<p<1$} \label{sec:TVnonconvex}
To motivate this regularization, we consider first $\ell_0$-minimization, where for $u\in\mathbb R^n$
$$
\|u\|_0= \textrm{ the number of non-zero entries in $u$},
$$
which is designed to promote sparsity in $u$. Solving the $\ell_0$ problem problem is in general NP-hard and therefore its convex relaxation, namely $\ell_1$-minimization, is considered in most sparse reconstruction approaches \cite{candes2006robust}. In this context, TV regularization can be seen as a convex approximation to a regularization which promotes sparsity of the gradient. Several papers indicate, however, that interesting regularization effects can be observed when studying regularizers that are $\ell_p$ norms in between $\ell_0$ and $\ell_1$, namely penalties of the form
$$
\|u\|_p^p = \sum |u_{i,j}|^p, \textrm{ with } 0<p<1,
$$  
compare \cite{nikolova2010fast} for instance. Recently this consideration has been extended to the case of the gradient in \cite{hintermüller2013nonconvex} where the authors study TV$^p$ regularization, that is 
$$
TV^p(u) = \|\nabla u\|_p^p = \int_\Omega \left(\left(\frac{\partial u}{\partial x}\right)^2
    + \left( \frac{\partial u}{\partial y}\right)^2\right)^{p/2}~ d(x,y), \textrm{ with } 0<p<1,
$$
Analogous to before, for discrete $u$ we consider the discretized and smoothed TV$^p$ functional
\begin{equation} \label{TVpfunctgray}
     J(u)=\Delta x \Delta y \sum_{i=1}^{N_x} \sum_{j=1}^{N_y} \psi \left( (D^x_{ij}u)^2 
     + (D^y_{ij}u)^2 \right),
\end{equation}
with $D^x$, $D^y$ given as in Section \ref{sec:TVdenoising} and $\psi(t)=(t+\beta)^{p/2}$ with $0<p<1$, and its corresponding denoising functional
$$
    V_\alpha(u)=\frac{1}{2} \Delta x \Delta y \sum_{i=1}^{Nx} \sum_{j=1}^{Ny} \left(u_{i,j} - (u_0)_{i,j}\right)^2+ \alpha J(u).
$$
As in Subsection \ref{sec:TVinpainting} we employ the Itoh \& Abe discrete gradient with adaptive step size selection. In Figure \ref{fig:legodenoiseTVp} we show a de-noising result with TV$^p$ regularization \eqref{TVpfunctgray} and $p=0.8$ and $p=0.2$. Since $V_\alpha$ is non-convex this time, we consider the behavior of the discrete gradient flow for two different initializations. We  initialize the discrete gradient flow once with the noisy image $u_0$ and once with a random initialization (randomly choosing the intensity in every pixel of the initial state). For both initializations the discrete gradient flow seems to converge to a decent critical point of $V_\alpha$, where $\alpha$ was chosen $0.05$ for $p=0.8$ and $\alpha=0.5$ for $p=0.2$. In fact, in both cases both critical points seem to converge to a similar energy level, compare Figure \ref{fig:legodenoiseTVpenerg}. Note also, as $p$ decreases, the gradient of the image at the computed minimum becomes sparser.

\begin{figure}[h!]
\centering
\includegraphics[width=0.4\textwidth]{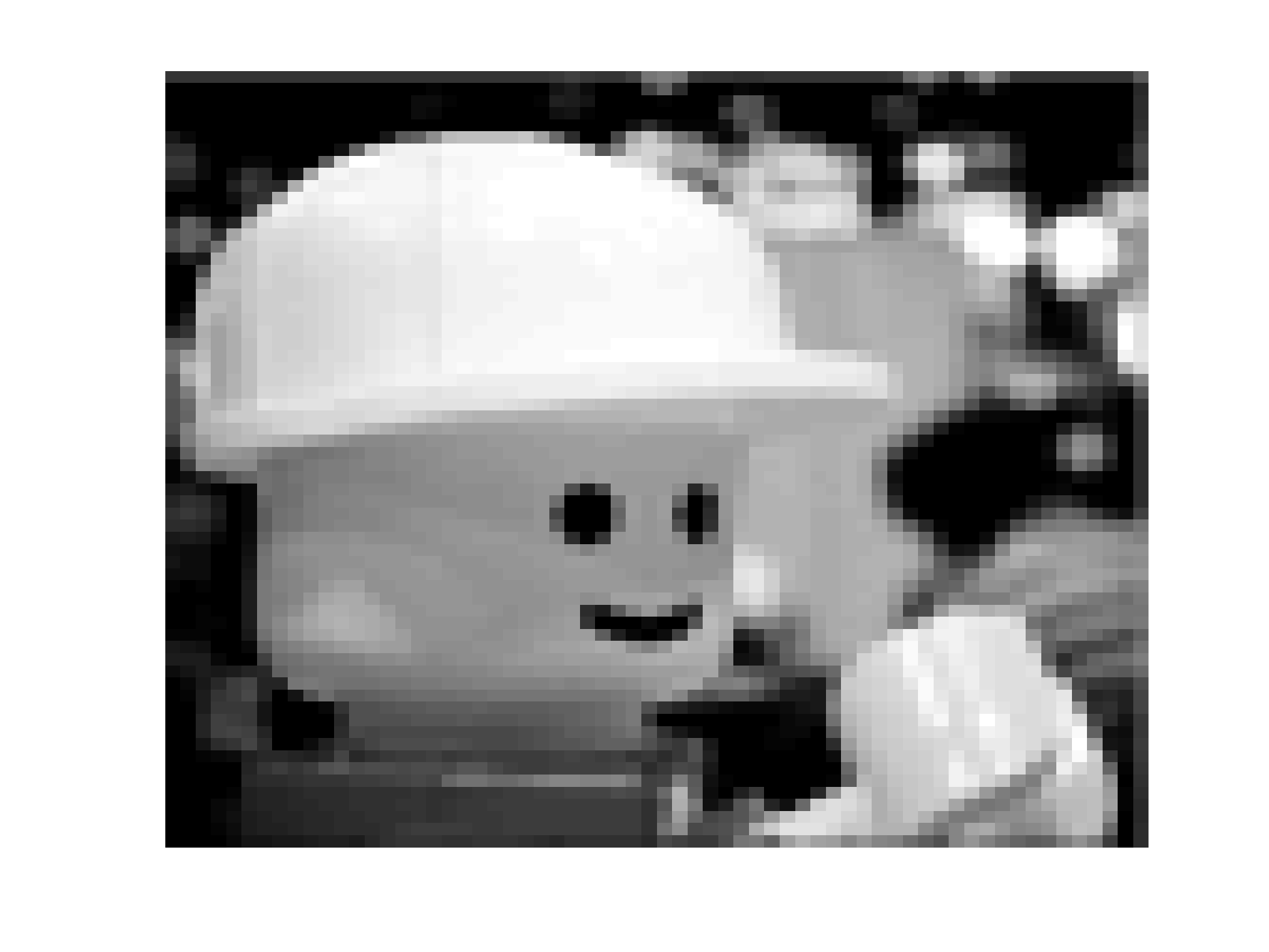}
\includegraphics[width=0.4\textwidth]{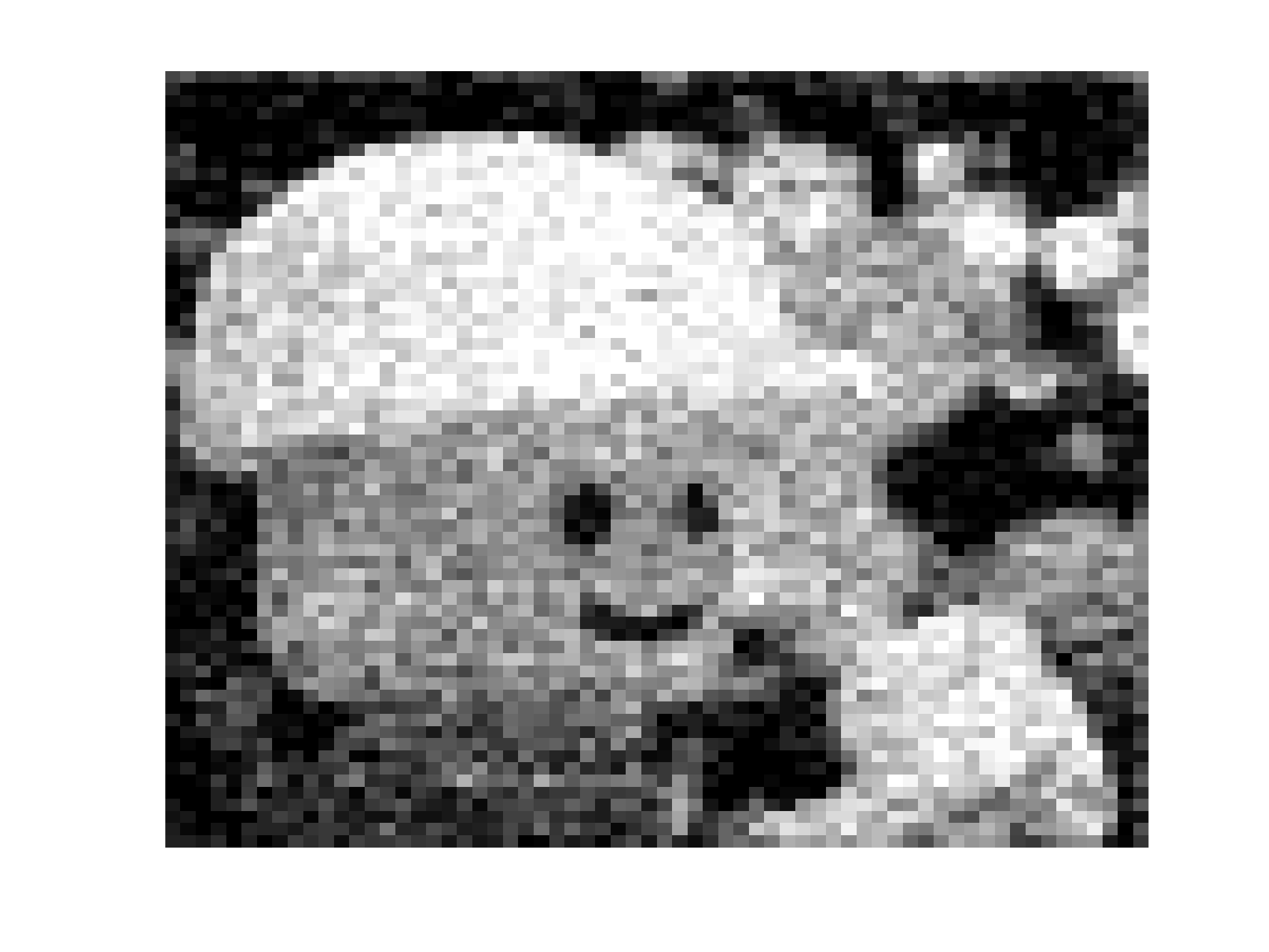}\\
\includegraphics[width=0.4\textwidth]{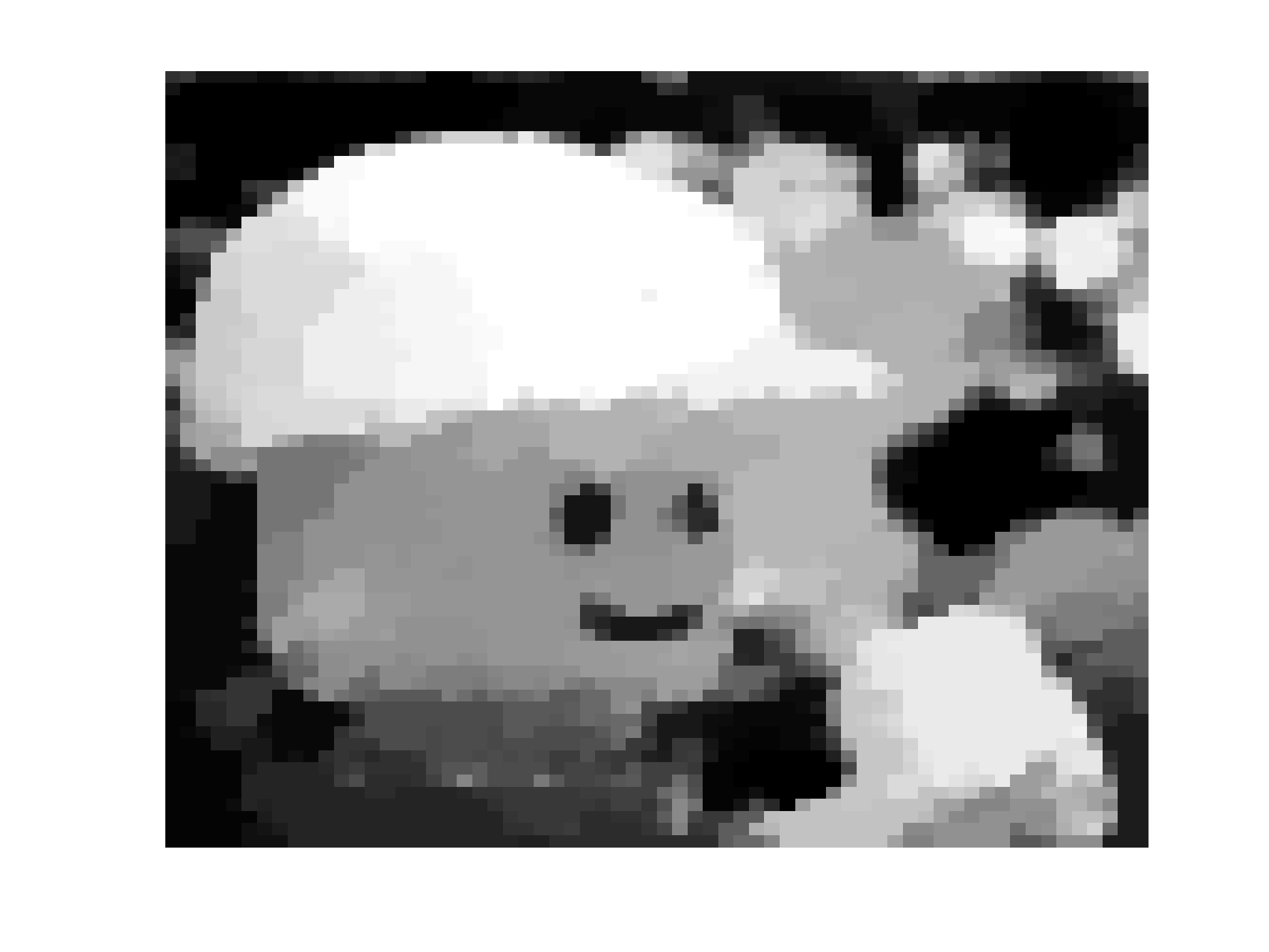}
\includegraphics[width=0.4\textwidth]{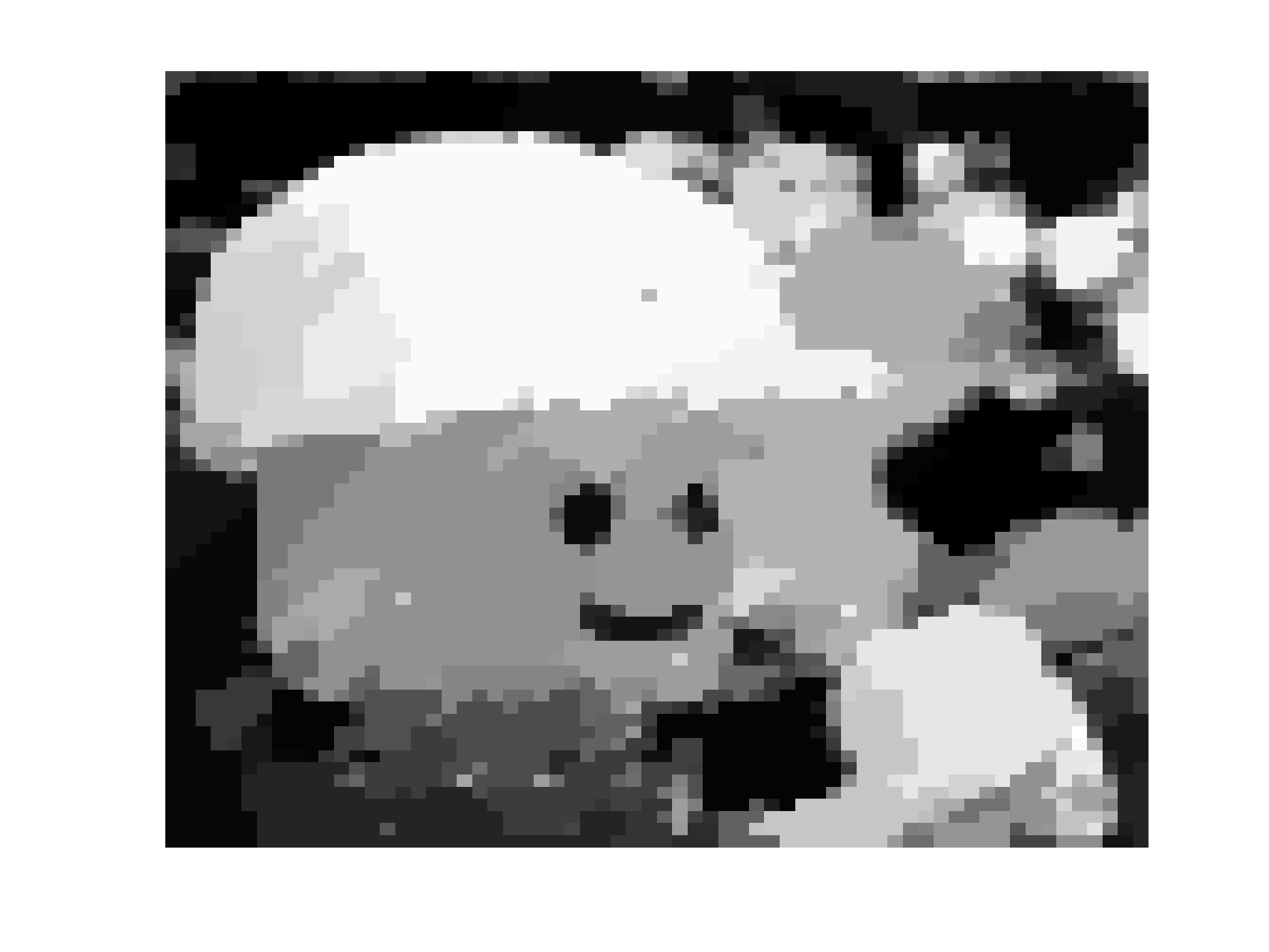}\\
\includegraphics[width=0.4\textwidth]{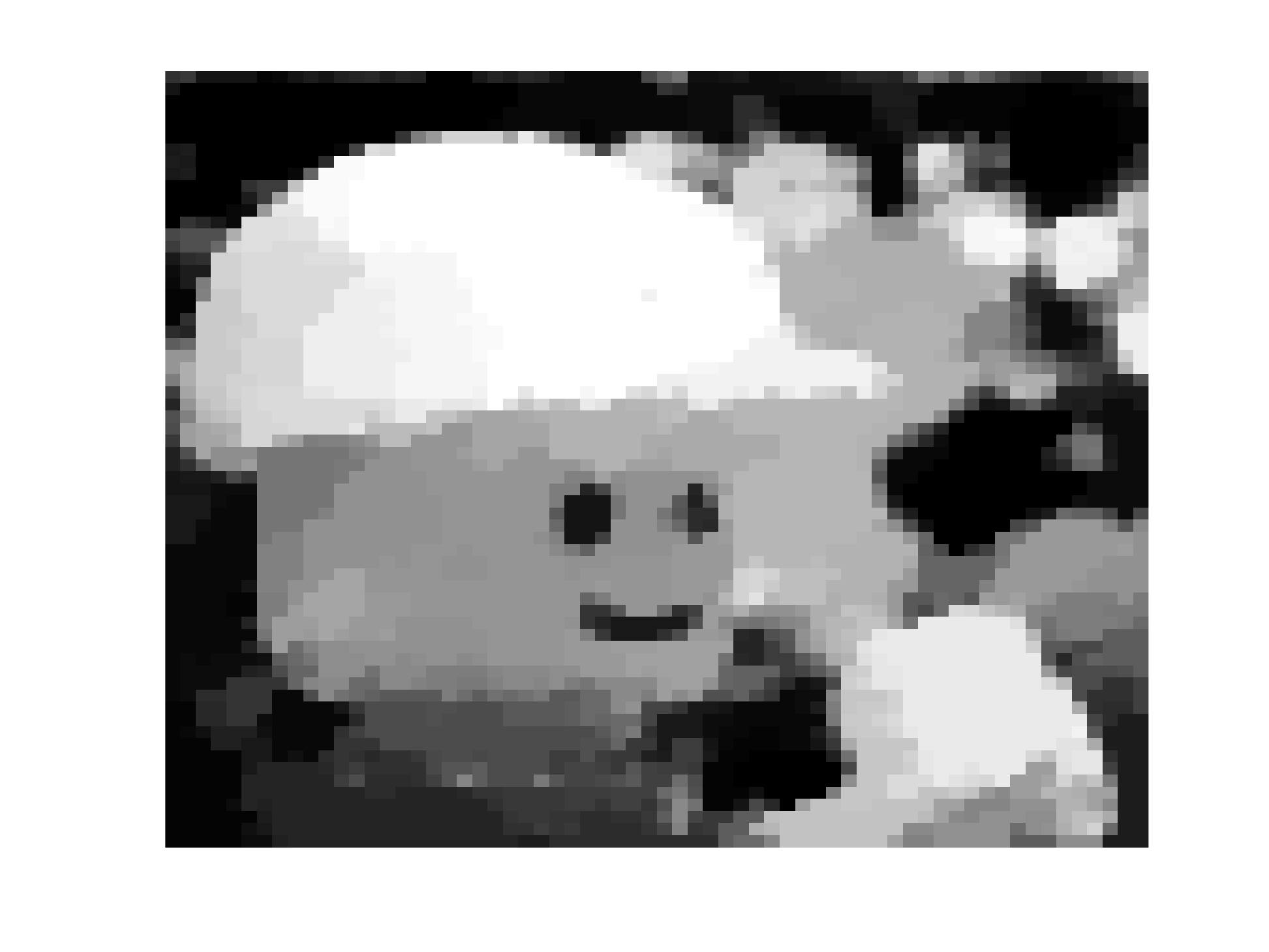}
\includegraphics[width=0.4\textwidth]{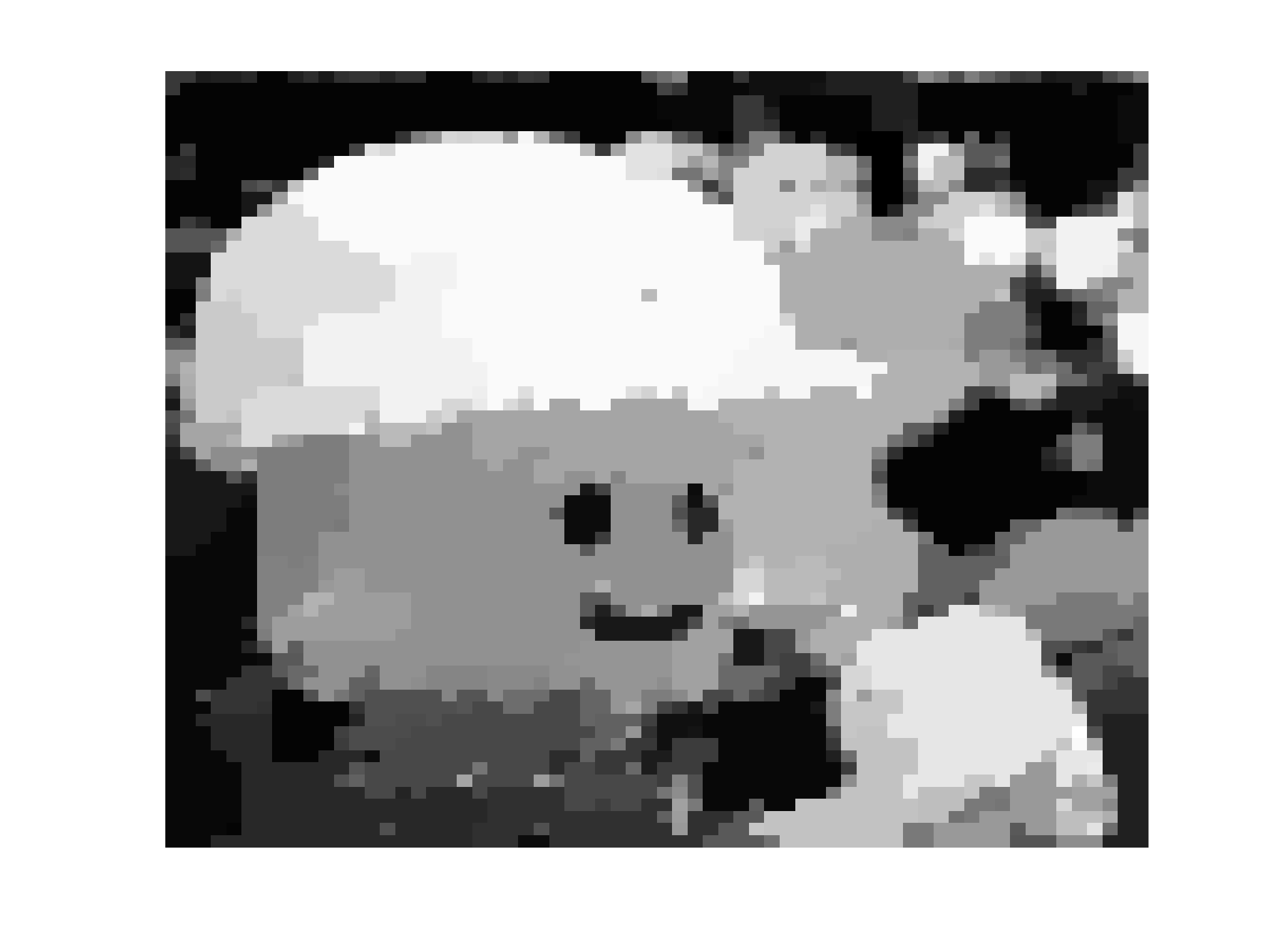}
\caption{TV$^p$ denoising with the Itoh \& Abe discrete gradient for $p=0.8$ (first column) and $p=0.2$ (second column). First row: Original and noisy image; second row: TV$^p$ denoising result with noisy image $u_0$ as initial condition; third row: TV$^p$ denoising result with random initial condition.}
\label{fig:legodenoiseTVp}
\end{figure}

\begin{figure}[h!]
\begin{center}
\hspace{-1cm}
{\tiny
\begin{tikzpicture}
\begin{axis}[width=7.5cm,height=5cm,xmin=0.0,xmax=20.5,ymin=0,ymax=70, xtick={0,2,4,6,8,10,12,14,16,18,20}, ytick={0,10,20,30,40,50,60,70}, xticklabels={$0$, $2$, $4$, $6$, $8$, $10$, $12$, $14$, $16$, $18$, $20$},  yticklabels={$0$, $10$, $20$, $30$, $40$, $50$, $60$, $70$}, title={}, xlabel=\empty, ylabel=\empty]
\addplot[thick,color=red] coordinates {
( 1, 7.1581 )
( 2, 5.2183 )
( 3, 4.7627 )
( 4, 4.5492 )
( 5, 4.4468 )
( 6, 4.3979 )
( 7, 4.3748 )
( 8, 4.3595 )
( 9, 4.3495 )
( 10, 4.3426 )
( 11, 4.3366 )
( 12, 4.3317 )
( 13, 4.327 )
( 14, 4.3227 )
( 15, 4.3195 )
( 16, 4.3169 )
( 17, 4.315 )
( 18, 4.3134 )
( 19, 4.3119 )
( 20, 4.3105 )
};
\addplot[thick,color=blue] coordinates {
( 1, 63.6871 )
( 2, 7.2108 )
( 3, 5.1443 )
( 4, 4.802 )
( 5, 4.6189 )
( 6, 4.4935 )
( 7, 4.4212 )
( 8, 4.379 )
( 9, 4.3585 )
( 10, 4.3487 )
( 11, 4.3422 )
( 12, 4.3366 )
( 13, 4.3319 )
( 14, 4.3279 )
( 15, 4.3244 )
( 16, 4.3213 )
( 17, 4.3187 )
( 18, 4.3165 )
( 19, 4.3146 )
( 20, 4.3131 )
};
\draw[red, thick]  (axis cs: 10,62) -- (axis cs: 11.5,62);
\draw[blue, thick] (axis cs: 10,55) -- (axis cs: 11.5,55);
\node[right] at (axis cs: 12, 62) {initial noisy image};
\node[right] at (axis cs: 12, 55) {initial random};
\end{axis}
\end{tikzpicture} 
\begin{tikzpicture}
\begin{axis}[width=7.5cm,height=5cm,xmin=-5,xmax=205,ymin=0,ymax=7, xtick={0,20,40,60,80,100,120,140,160,180,200}, ytick={1, 2, 3, 4, 5, 6, 7}, xticklabels={$0$, $20$, $40$, $60$, $80$, $100$, $120$, $140$, $160$, $180$, $200$},  yticklabels={$1$, $2$, $3$, $4$, $5$, $6$, $7$}, title={}, xlabel=\empty, ylabel=\empty]
\addplot[thick,color=red] coordinates {
( 1, 1.5765 )
( 2, 1.4895 )
( 3, 1.4194 )
( 4, 1.3827 )
( 5, 1.3575 )
( 6, 1.3344 )
( 7, 1.3142 )
( 8, 1.2959 )
( 9, 1.2812 )
( 10, 1.2706 )
( 11, 1.2641 )
( 12, 1.2607 )
( 13, 1.2585 )
( 14, 1.2569 )
( 15, 1.2555 )
( 16, 1.2538 )
( 17, 1.2521 )
( 18, 1.2503 )
( 19, 1.2484 )
( 20, 1.2464 )
( 21, 1.2444 )
( 22, 1.2423 )
( 23, 1.2403 )
( 24, 1.2384 )
( 25, 1.2367 )
( 26, 1.2356 )
( 27, 1.2346 )
( 28, 1.2335 )
( 29, 1.2325 )
( 30, 1.2318 )
( 31, 1.2311 )
( 32, 1.2302 )
( 33, 1.2288 )
( 34, 1.2269 )
( 35, 1.2248 )
( 36, 1.2221 )
( 37, 1.2181 )
( 38, 1.2139 )
( 39, 1.2118 )
( 40, 1.2096 )
( 41, 1.2067 )
( 42, 1.2032 )
( 43, 1.1988 )
( 44, 1.1934 )
( 45, 1.1875 )
( 46, 1.1819 )
( 47, 1.1777 )
( 48, 1.1753 )
( 49, 1.1743 )
( 50, 1.1737 )
( 51, 1.1732 )
( 52, 1.1728 )
( 53, 1.1723 )
( 54, 1.1719 )
( 55, 1.1713 )
( 56, 1.171 )
( 57, 1.1707 )
( 58, 1.1703 )
( 59, 1.17 )
( 60, 1.1697 )
( 61, 1.1695 )
( 62, 1.1693 )
( 63, 1.1691 )
( 64, 1.1688 )
( 65, 1.1684 )
( 66, 1.168 )
( 67, 1.1676 )
( 68, 1.1673 )
( 69, 1.167 )
( 70, 1.1668 )
( 71, 1.1665 )
( 72, 1.1662 )
( 73, 1.1659 )
( 74, 1.1657 )
( 75, 1.1655 )
( 76, 1.1652 )
( 77, 1.1649 )
( 78, 1.1647 )
( 79, 1.1644 )
( 80, 1.1643 )
( 81, 1.1641 )
( 82, 1.1638 )
( 83, 1.1635 )
( 84, 1.1632 )
( 85, 1.1628 )
( 86, 1.1624 )
( 87, 1.162 )
( 88, 1.1616 )
( 89, 1.1613 )
( 90, 1.161 )
( 91, 1.1608 )
( 92, 1.1606 )
( 93, 1.1604 )
( 94, 1.1603 )
( 95, 1.1602 )
( 96, 1.1601 )
( 97, 1.1599 )
( 98, 1.1598 )
( 99, 1.1597 )
( 100, 1.1595 )
( 101, 1.1594 )
( 102, 1.1593 )
( 103, 1.1591 )
( 104, 1.159 )
( 105, 1.159 )
( 106, 1.1589 )
( 107, 1.1588 )
( 108, 1.1587 )
( 109, 1.1587 )
( 110, 1.1586 )
( 111, 1.1585 )
( 112, 1.1584 )
( 113, 1.1583 )
( 114, 1.1582 )
( 115, 1.1581 )
( 116, 1.158 )
( 117, 1.1579 )
( 118, 1.1578 )
( 119, 1.1577 )
( 120, 1.1575 )
( 121, 1.1573 )
( 122, 1.1572 )
( 123, 1.1571 )
( 124, 1.157 )
( 125, 1.1568 )
( 126, 1.1565 )
( 127, 1.1562 )
( 128, 1.1559 )
( 129, 1.1557 )
( 130, 1.1556 )
( 131, 1.1555 )
( 132, 1.1554 )
( 133, 1.1553 )
( 134, 1.1553 )
( 135, 1.1552 )
( 136, 1.1551 )
( 137, 1.155 )
( 138, 1.1549 )
( 139, 1.1548 )
( 140, 1.1547 )
( 141, 1.1545 )
( 142, 1.1542 )
( 143, 1.154 )
( 144, 1.1534 )
( 145, 1.1529 )
( 146, 1.1522 )
( 147, 1.1515 )
( 148, 1.1511 )
( 149, 1.1509 )
( 150, 1.1508 )
( 151, 1.1506 )
( 152, 1.1504 )
( 153, 1.1501 )
( 154, 1.1496 )
( 155, 1.1489 )
( 156, 1.148 )
( 157, 1.1468 )
( 158, 1.1452 )
( 159, 1.1432 )
( 160, 1.1411 )
( 161, 1.1395 )
( 162, 1.1384 )
( 163, 1.1379 )
( 164, 1.1378 )
( 165, 1.1377 )
( 166, 1.1376 )
( 167, 1.1375 )
( 168, 1.1375 )
( 169, 1.1374 )
( 170, 1.1373 )
( 171, 1.1372 )
( 172, 1.1371 )
( 173, 1.1371 )
( 174, 1.1371 )
( 175, 1.137 )
( 176, 1.137 )
( 177, 1.137 )
( 178, 1.1369 )
( 179, 1.1369 )
( 180, 1.1369 )
( 181, 1.1369 )
( 182, 1.1368 )
( 183, 1.1368 )
( 184, 1.1368 )
( 185, 1.1367 )
( 186, 1.1367 )
( 187, 1.1367 )
( 188, 1.1367 )
( 189, 1.1367 )
( 190, 1.1366 )
( 191, 1.1366 )
( 192, 1.1365 )
( 193, 1.1364 )
( 194, 1.1363 )
( 195, 1.1363 )
( 196, 1.1362 )
( 197, 1.1362 )
( 198, 1.1361 )
( 199, 1.1361 )
( 200, 1.136 )
};
\addplot[thick,color=blue] coordinates {
( 1, 6.5346 )
( 2, 4.5758 )
( 3, 2.5777 )
( 4, 1.5865 )
( 5, 1.4489 )
( 6, 1.3926 )
( 7, 1.3526 )
( 8, 1.3239 )
( 9, 1.2999 )
( 10, 1.2795 )
( 11, 1.2636 )
( 12, 1.2512 )
( 13, 1.2431 )
( 14, 1.2391 )
( 15, 1.2365 )
( 16, 1.2345 )
( 17, 1.2328 )
( 18, 1.2308 )
( 19, 1.2289 )
( 20, 1.2274 )
( 21, 1.2261 )
( 22, 1.2247 )
( 23, 1.2234 )
( 24, 1.2223 )
( 25, 1.2214 )
( 26, 1.2204 )
( 27, 1.2191 )
( 28, 1.2178 )
( 29, 1.2165 )
( 30, 1.2155 )
( 31, 1.2145 )
( 32, 1.2134 )
( 33, 1.2121 )
( 34, 1.2108 )
( 35, 1.2096 )
( 36, 1.2087 )
( 37, 1.2078 )
( 38, 1.207 )
( 39, 1.2062 )
( 40, 1.2056 )
( 41, 1.2051 )
( 42, 1.2044 )
( 43, 1.2032 )
( 44, 1.2021 )
( 45, 1.2009 )
( 46, 1.1995 )
( 47, 1.1985 )
( 48, 1.1978 )
( 49, 1.1972 )
( 50, 1.1965 )
( 51, 1.1958 )
( 52, 1.1952 )
( 53, 1.1946 )
( 54, 1.194 )
( 55, 1.1934 )
( 56, 1.193 )
( 57, 1.1926 )
( 58, 1.1921 )
( 59, 1.1916 )
( 60, 1.1912 )
( 61, 1.1908 )
( 62, 1.1903 )
( 63, 1.1898 )
( 64, 1.1894 )
( 65, 1.189 )
( 66, 1.1886 )
( 67, 1.1882 )
( 68, 1.1877 )
( 69, 1.1872 )
( 70, 1.1869 )
( 71, 1.1866 )
( 72, 1.1862 )
( 73, 1.1859 )
( 74, 1.1856 )
( 75, 1.1853 )
( 76, 1.185 )
( 77, 1.1846 )
( 78, 1.1838 )
( 79, 1.1824 )
( 80, 1.1808 )
( 81, 1.179 )
( 82, 1.1774 )
( 83, 1.1763 )
( 84, 1.1749 )
( 85, 1.1732 )
( 86, 1.1708 )
( 87, 1.168 )
( 88, 1.1653 )
( 89, 1.1631 )
( 90, 1.1619 )
( 91, 1.1613 )
( 92, 1.161 )
( 93, 1.1607 )
( 94, 1.1604 )
( 95, 1.1602 )
( 96, 1.1601 )
( 97, 1.1599 )
( 98, 1.1596 )
( 99, 1.1594 )
( 100, 1.1592 )
( 101, 1.159 )
( 102, 1.1588 )
( 103, 1.1586 )
( 104, 1.1584 )
( 105, 1.1582 )
( 106, 1.1579 )
( 107, 1.1576 )
( 108, 1.1574 )
( 109, 1.1572 )
( 110, 1.157 )
( 111, 1.1567 )
( 112, 1.1565 )
( 113, 1.1563 )
( 114, 1.156 )
( 115, 1.1556 )
( 116, 1.1553 )
( 117, 1.1548 )
( 118, 1.1545 )
( 119, 1.1543 )
( 120, 1.154 )
( 121, 1.1537 )
( 122, 1.1534 )
( 123, 1.1531 )
( 124, 1.1529 )
( 125, 1.1528 )
( 126, 1.1527 )
( 127, 1.1526 )
( 128, 1.1524 )
( 129, 1.1523 )
( 130, 1.1522 )
( 131, 1.1522 )
( 132, 1.1521 )
( 133, 1.1519 )
( 134, 1.1518 )
( 135, 1.1515 )
( 136, 1.1508 )
( 137, 1.1504 )
( 138, 1.1498 )
( 139, 1.1492 )
( 140, 1.1486 )
( 141, 1.1484 )
( 142, 1.1482 )
( 143, 1.1478 )
( 144, 1.1474 )
( 145, 1.1469 )
( 146, 1.1463 )
( 147, 1.1453 )
( 148, 1.144 )
( 149, 1.1425 )
( 150, 1.1409 )
( 151, 1.1395 )
( 152, 1.1387 )
( 153, 1.1383 )
( 154, 1.1382 )
( 155, 1.1381 )
( 156, 1.1381 )
( 157, 1.138 )
( 158, 1.1379 )
( 159, 1.1379 )
( 160, 1.1378 )
( 161, 1.1378 )
( 162, 1.1377 )
( 163, 1.1376 )
( 164, 1.1376 )
( 165, 1.1375 )
( 166, 1.1374 )
( 167, 1.1374 )
( 168, 1.1374 )
( 169, 1.1374 )
( 170, 1.1373 )
( 171, 1.1373 )
( 172, 1.1372 )
( 173, 1.1372 )
( 174, 1.1372 )
( 175, 1.1371 )
( 176, 1.1371 )
( 177, 1.1371 )
( 178, 1.137 )
( 179, 1.137 )
( 180, 1.137 )
( 181, 1.1369 )
( 182, 1.1369 )
( 183, 1.1369 )
( 184, 1.1368 )
( 185, 1.1368 )
( 186, 1.1368 )
( 187, 1.1368 )
( 188, 1.1368 )
( 189, 1.1368 )
( 190, 1.1367 )
( 191, 1.1367 )
( 192, 1.1367 )
( 193, 1.1367 )
( 194, 1.1367 )
( 195, 1.1367 )
( 196, 1.1367 )
( 197, 1.1367 )
( 198, 1.1366 )
( 199, 1.1366 )
( 200, 1.1366 )
};
\draw[red, thick]  (axis cs: 100,6.2) -- (axis cs: 115,6.2);
\draw[blue, thick] (axis cs: 100,5.5) -- (axis cs: 115,5.5);
\node[right] at (axis cs: 120, 6.2) {initial noisy image};
\node[right] at (axis cs: 120, 5.5) {initial random};
\end{axis}
\end{tikzpicture} 
}
\end{center}
\caption{TV$^p$ denoising with the Itoh \& Abe discrete gradient: energy decrease for the result in Figure~\ref{fig:legodenoiseTVp}.}
\label{fig:legodenoiseTVpenerg}
\end{figure}
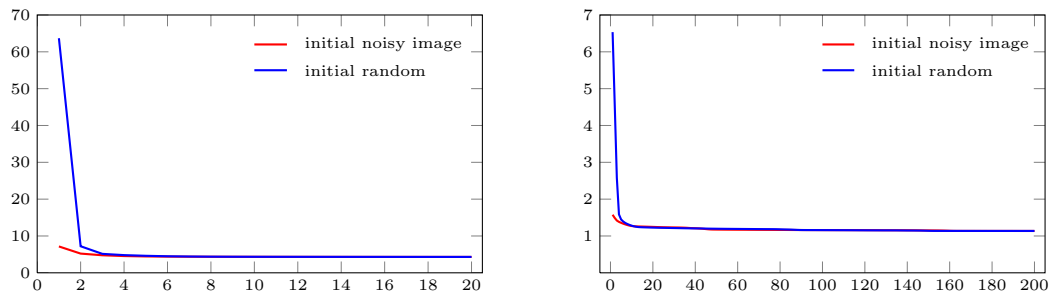

\section{Conclusion} \label{sec:conclusion}
We discussed discrete gradient methods, well-known in Geometric Numerical Integration for the preservation of dissipation in variational equations, with respect to their use in image processing. We assumed that $V$ is smooth which is sufficient when it comes to actual computations. However, preliminary considerations by the standard techniques suggest this assumption may be weakened significantly. Note also, that in this paper we consider gradient flows of $V$ with respect to the Euclidean inner product only. This can be however generalized, cf. \cite{AVF12}, to gradient flows with respect to other inner products as they appear in image processing such as $H^{-1}$ gradient flows \cite{burger2009cahn,osher2003image} or Wasserstein gradient flows \cite{benning2013primal,burger2012regularized,during2012high,ferradans2014regularized,lellmann2014imaging,peyre2012wasserstein,rabin2011wasserstein,schmitzer2013object}, just to name a few. We believe that the presented theory, that guarantees the convergence to the equilibrium of any discrete gradient method for a wide range of functionals and gradient flows used in image processing, as well as the conducted experiments indicate that discrete gradient methods could be very interesting for image processing tasks.  

\bibliographystyle{plain} 
\bibliography{lit}

\begin{thebibliography}{10}

\bibitem{AcVog94}
R.~Acar and C.~R. Vogel.
\newblock Analysis of bounded variation penalty methods for ill-posed problems.
\newblock {\em Inverse Problems}, 10(6):1217--1229, 1994.

\bibitem{benning2013primal}
M.~Benning, L.~Calatroni, B.~D{\"u}ring, and C.-B. Sch{\"o}nlieb.
\newblock A primal-dual approach for a total variation {W}asserstein flow.
\newblock In {\em Geometric Science of Information}, pages 413--421. Springer,
  2013.

\bibitem{bertozzi2007inpainting}
A.~L. Bertozzi, S.~Esedoglu, and A.~Gillette.
\newblock Inpainting of binary images using the {Cahn-Hilliard} equation.
\newblock {\em IEEE Transactions on image processing}, 16(1):285--291, 2007.

\bibitem{bertozzi2012diffuse}
A.~L. Bertozzi and A.~Flenner.
\newblock Diffuse interface models on graphs for classification of high
  dimensional data.
\newblock {\em Multiscale Modeling \& Simulation}, 10(3):1090--1118, 2012.

\bibitem{BloChan98}
P.~Blomgren and T.~F. Chan.
\newblock Total variation methods for restoration of vecor-valued images.
\newblock {\em IEEE Trans. Image Process}, 7:304--309, 1998.

\bibitem{burger2012regularized}
M.~Burger, M.~Franek, and C.-B. Sch{\"o}nlieb.
\newblock Regularized regression and density estimation based on optimal
  transport.
\newblock {\em Applied Mathematics Research eXpress}, 2012(2):209--253, 2012.

\bibitem{burger2009cahn}
M.~Burger, L.~He, and C.-B. Sch{\"o}nlieb.
\newblock {Cahn-Hilliard} inpainting and a generalization for grayvalue images.
\newblock {\em SIAM Journal on Imaging Sciences}, 2(4):1129--1167, 2009.

\bibitem{candes2006robust}
E.~J. Cand{\`e}s, J.~Romberg, and T.~Tao.
\newblock {Robust uncertainty principles: Exact signal reconstruction from
  highly incomplete frequency information}.
\newblock {\em Information Theory, IEEE Transactions on}, 52(2):489--509, 2006.

\bibitem{ChamLio97}
A.~Chambolle and P.~L. Lions.
\newblock Image recovery via total variational minimization and related
  problems.
\newblock {\em Numer. Math.}, 76:167--188, 1997.

\bibitem{chan1999convergence}
T.~F. Chan and P.~Mulet.
\newblock On the convergence of the lagged diffusivity fixed point method in
  total variation image restoration.
\newblock {\em SIAM Journal on Numerical Analysis}, 36(2):354--367, 1999.

\bibitem{ChanShenbook05}
T.~F. Chan and J.~Shen.
\newblock {\em Image processing and analysis: Variational, PDE, wavelet, and
  stochastic methods}.
\newblock Society for Industrial and Applied Mathematics (SIAM), Philadelphia,
  PA, 2005.

\bibitem{ChanVese01}
T.~F. Chan and L.~A. Vese.
\newblock Active contours without edges.
\newblock {\em IEEE Trans. Image Process.}, 10(2):266--277, 2001.

\bibitem{Cie13}
J.~L. Cie{\'s}li{\'n}ski.
\newblock Locally exact modifications of discrete gradient schemes.
\newblock {\em Phys. Lett. A}, 377(8):592--597, 2013.

\bibitem{CoHai11}
D.~Cohen and E.~Hairer.
\newblock Linear energy-preserving integrators for {P}oisson systems.
\newblock {\em BIT}, 51(1):91--101, 2011.

\bibitem{during2012high}
B.~D{\"u}ring and C.-B. Sch{\"o}nlieb.
\newblock A high-contrast fourth-order pde from imaging: numerical solution by
  {ADI} splitting.
\newblock {\em Multi-scale and High-Contrast Partial Differential Equations, H.
  Ammari et al.(eds.)}, pages 93--103, 2012.

\bibitem{AVF12}
{E.~Celledoni, V.~Grimm, R.~I.~McLachlan, D.~I.~McLaren, D.~O'Neale, B.~Owren
  and G.~R.~W.~Quispel}.
\newblock Preserving energy resp. dissipation in numerical {PDE}s using the
  ``average vector field'' method.
\newblock {\em J. Comput. Phys.}, 231(20):6770--6789, 2012.

\bibitem{ferradans2014regularized}
S.~Ferradans, N.~Papadakis, G.~Peyr{\'e}, and J.-F. Aujol.
\newblock Regularized discrete optimal transport.
\newblock {\em SIAM Journal on Imaging Sciences}, 7(3):1853--1882, 2014.

\bibitem{Gonzales96}
O.~Gonzalez.
\newblock Time integration and discrete {H}amiltonian systems.
\newblock {\em J. Nonlinear Sci.}, 6(5):449--467, 1996.

\bibitem{GriHeWi06}
V.~Grimm, S.~Henn, and K.~Witsch.
\newblock A higher-order {PDE}-based image registration approach.
\newblock {\em Numer. Linear Algebra Appl.}, 13(5):399--417, 2006.

\bibitem{HaiLub13}
E.~Hairer and Ch. Lubich.
\newblock {Energy-diminishing integration of gradient systems}.
\newblock {\em IMA Journal of Numerical Analysis}, 34:452--461, 2014.

\bibitem{GNI2}
E.~Hairer, Ch. Lubich, and G.~Wanner.
\newblock {\em Geometric numerical integration}, volume~31 of {\em Springer
  Series in Computational Mathematics}.
\newblock Springer, Heidelberg, 2010.
\newblock Second Edition.

\bibitem{HNLdeblur06}
P.~C. Hansen, J.~G. Nagy, and D.~P. O'Leary.
\newblock {\em Deblurring images; Matrices, spectra, and filtering}, volume~3
  of {\em Fundamentals of Algorithms}.
\newblock Society for Industrial and Applied Mathematics (SIAM), Philadelphia,
  PA, 2006.

\bibitem{HeStiCG}
M.~R. Hestenes and E.~Stiefel.
\newblock Methods of conjugate gradients for solving linear systems.
\newblock {\em J. Research Nat. Bur. Standards}, 49:409--436 (1953), 1952.

\bibitem{hintermüller2013nonconvex}
M.~Hinterm{\"u}ller and T.~Wu.
\newblock Nonconvex {$\mbox{TV}^q$}-models in image restoration: Analysis and a
  trust-region regularization--based superlinearly convergent solver.
\newblock {\em SIAM Journal on Imaging Sciences}, 6(3):1385--1415, 2013.

\bibitem{ItohAbe88}
T.~Itoh and K.~Abe.
\newblock Hamiltonian-conserving discrete canonical equations based on
  variational difference quotients.
\newblock {\em J. Comput. Phys.}, 76(1):85--102, 1988.

\bibitem{lellmann2014imaging}
J.~Lellmann, D.~A. Lorenz, C.-B. Schönlieb, and T.~Valkonen.
\newblock {Imaging with Kantorovich--Rubinstein Discrepancy}.
\newblock {\em SIAM Journal on Imaging Sciences}, 7(4):2833--2859, 2014.

\bibitem{MatFuri01}
T.~Matsuo and D.~Furihata.
\newblock Dissipative or conservative finite-difference schemes for
  complex-valued nonlinear partial differential equations.
\newblock {\em J. Comput. Phys.}, 171(2):425--447, 2001.

\bibitem{SixLec}
R.~McLachlan and R.~Quispel.
\newblock Six lectures on the geometric integration of {ODE}s.
\newblock In {\em Foundations of computational mathematics ({O}xford, 1999)},
  volume 284 of {\em London Math. Soc. Lecture Note Ser.}, pages 155--210.
  Cambridge Univ. Press, Cambridge, 2001.

\bibitem{McLachlanQusipel99}
R.~I. McLachlan, G.~R.~W. Quispel, and N.~Robidoux.
\newblock Geometric integration using discrete gradients.
\newblock {\em R. Soc. Lond. Philos. Trans. Ser. A Math. Phys. Eng. Sci.},
  357(1754):1021--1045, 1999.

\bibitem{MiRatTan07}
O.~Michailovich, Y.~Rathi, and A.~Tannenbaum.
\newblock Image segmentation using active contours driven by the
  {B}hattacharyya gradient flow.
\newblock {\em IEEE Trans. Image Process.}, 16(11):2787--2801, 2007.

\bibitem{nikolova2010fast}
M.~Nikolova, M.~K. Ng, and C.-P. Tam.
\newblock Fast nonconvex nonsmooth minimization methods for image restoration
  and reconstruction.
\newblock {\em Image Processing, IEEE Transactions on}, 19(12):3073--3088,
  2010.

\bibitem{osher2003image}
S.~Osher, A.~Sol{\'e}, and L.~Vese.
\newblock Image decomposition and restoration using total variation
  minimization and the {H{$^{-1}$}}.
\newblock {\em Multiscale Modeling \& Simulation}, 1(3):349--370, 2003.

\bibitem{PMscale}
P.~Perona and J.~Malik.
\newblock Scale-space and edge detection using anisotropic diffusion.
\newblock {\em IEEE Transactions on Pattern Analysis and Machine Intelligence},
  12:629--639, 1990.

\bibitem{peyre2012wasserstein}
G.~Peyr{\'e}, J.~Fadili, and J.~Rabin.
\newblock Wasserstein active contours.
\newblock In {\em Image Processing (ICIP), 2012 19th IEEE International
  Conference on}, pages 2541--2544. IEEE, 2012.

\bibitem{QuiTu96}
G.~R.~W. Quispel and G.~S. Turner.
\newblock Discrete gradient methods for solving {ODE}s numerically while
  preserving a first integral.
\newblock {\em J. Phys. A}, 29(13):L341--L349, 1996.

\bibitem{rabin2011wasserstein}
J.~Rabin, G.~Peyr{\'e}, J.~Delon, and M.~Bernot.
\newblock Wasserstein barycenter and its application to texture mixing.
\newblock In {\em Scale Space and Variational Methods in Computer Vision},
  pages 435--446. Springer, 2011.

\bibitem{RuOshFa92}
L.~I. Rudin, S.~Osher, and E.~Fatemi.
\newblock Nonlinear total variation based noise removal algorithms.
\newblock {\em Physica D: Nonlinear Phenomena}, 60:259--268, 1992.

\bibitem{schmitzer2013object}
B.~Schmitzer and Ch. Schn{\"o}rr.
\newblock Object segmentation by shape matching with {W}asserstein modes.
\newblock In {\em Energy Minimization Methods in Computer Vision and Pattern
  Recognition}, pages 123--136. Springer, 2013.

\bibitem{StrDrRu04}
R.~Strzodka, M.~Droske, and M.~Rumpf.
\newblock Image registration by a regularized gradient flow. {A} streaming
  implementation in {DX}9 graphics hardware.
\newblock {\em Computing}, 73(4):373--389, 2004.

\bibitem{StuHam96}
A.~M. Stuart and A.~R. Humphries.
\newblock {\em Dynamical systems and numerical analysis}, volume~2 of {\em
  Cambridge Monographs on Applied and Computational Mathematics}.
\newblock Cambridge University Press, Cambridge, 1996.

\bibitem{USC-SIPIbase}
{The USC-SIPI Image Database, available at:
  http://sipi.usc.edu/services/database/Database.html}.

\bibitem{Weickertbook98}
J.~Weickert.
\newblock {\em Anisotropic diffusion in image processing}.
\newblock European Consortium for Mathematics in Industry. B. G. Teubner,
  Stuttgart, 1998.

\bibitem{XuPri98}
Ch. Xu and J.~L. Prince.
\newblock Snakes, shapes, and gradient vector flow.
\newblock {\em IEEE Trans. Image Process.}, 7(3):359--369, 1998.

\end{thebibliography}

\end{document}